\documentclass[11pt,leqno]{article}
\usepackage{amsmath}
\usepackage{amsthm}
\usepackage{amstext}
\usepackage{amsopn}
\usepackage{texdraw}
\usepackage{graphicx}
\newtheorem{theorem}{Theorem}[section]
\newtheorem{lemma}[theorem]{Lemma}
\newtheorem{proposition}[theorem]{Proposition}

\theoremstyle{definition}
\newtheorem{definition}[theorem]{Definition}

\theoremstyle{remark}
\newtheorem{remark}[theorem]{Remark}

\begin{document}
\title{
 A new variational method with SPBC and  many stable choreographic solutions of the Newtonian $4$-body problem }
\author{Tiancheng Ouyang \\
 Department of Mathematics, Brigham Young University\\
 Provo, Utah 84602, USA\\
 Email: ouyang@math.byu.edu\\
 Zhifu Xie 
\\
Department of Mathematics and Computer Science\\
 Virginia State University\\
Petersburg, Virginia 23806, USA\\
 Email: zxie@vsu.edu}
\date{}

\maketitle
\begin{abstract}

   After the existence proof of the first remarkably stable simple choreographic
motion-- the figure eight of the planar three-body problem by Chenciner and Montgomery in 2000, a great number of simple choreographic solutions have been discovered numerically  but very few of them have rigorous existence proofs and none of them are stable. Most important to astronomy are stable periodic solutions which might actually be seen in some stellar system.  A question for simple choreographic solutions  on $n$-body problems naturally arises: Are there any other stable simple choreographic solutions except the figure eight? 

In this paper, we  prove the existence of infinitely many simple choreographic solutions in the classical Newtonian $4$-body problem by developing a new variational method with structural prescribed boundary conditions (SPBC). One of the essential features of the new method is that the method works for general $n$-body problem without any constraints on masses or symmetries and it is easy to execute numerically by a computation pragram. Surprisingly, a family of choreographic orbits of this type are all linearly stable. Among the many stable simple choreographic orbits,  the most extraordinary one is the stable star pentagon choreographic solution.  The star pentagon is assembled out of four pieces of curves which are obtained by minimizing the Lagrangian action functional over the SPBC.
   We also prove the existence of infinitely many double choreographic periodic solutions, infinitely many non-choreographic periodic solutions and uncountably many quasi-periodic solutions. Each type of periodic solutions have many stable solutions and possibly  infinitely many stable solutions. 

\end{abstract}
{\bf AMS classification number:} 37N05, 70F10,70H12, 70F15, 37N30, 70H05, \\

\section{ Introduction}

\begin{figure}
\includegraphics[height=7cm,width=.6\textwidth]{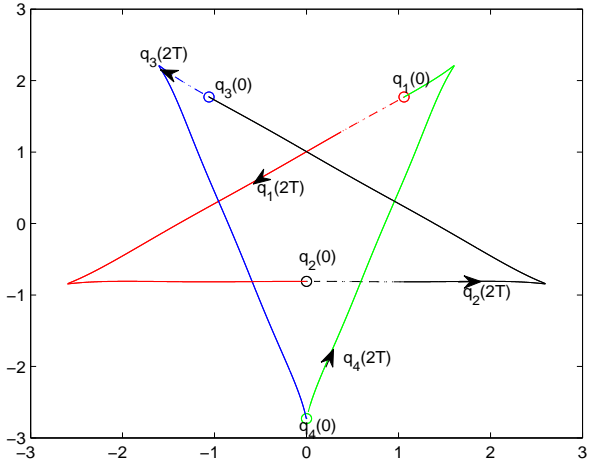}

\caption{ The first $40$-th of the star pentagon, traveling from an isosceles triangle $q_i(0)$ with one in the axis of its symmetry to a trapezoid, and then back to another isosceles triangle $q_i(2T)$. The entire star pentagon can be assembled by reflection, permutation, and rotation of the initial pieces (dashdotted). Initial conditions are: $q_1=(1.0598738926379,$ $   1.7699901770118 ),$ $ \dot{q}_1=(-0.55391384867197,  $ $    -0.39895079845794),$
 $q_2=(0, -0.80951135793043),$   $\dot{q}_2=(1.0936551555351,   0),$  $q_3=(-1.0598738926379, $ $  1.7699901770118), $    $\dot{q}_3=(-0.55391558212647,    $ $   0.39895379682134),$
     $q_4=(0,  -2.7304689960932),$       $\dot{q}_4=(0.01417427526245,   0),$       $m_1=m_2=m_3=m_4=1, T=1.$
  }\label{fig1}\end{figure}


Periodic solutions to the Newtonian $n$-body problem have been extensively studied for centuries. Variational method has been applied to  obtain solutions for the $n$-body problem more than one hundred years since  Poincar\'{e} \cite{PH} in 1896. 
In the past decade, the existence of many new interesting periodic orbits are proved by using variational method for the n-body problem. Most of them are found by minimizing the Lagrangian action on a symmetric loop space with some topological constraints (for example, see \cite{BT, BT2, Chen1, DengZhangZhou, FG, FT, TE, TV}).  

 A {\it simple choreographic solution} (for short, choreographic solution) is a periodic solution that all bodies chase one another along a single closed orbit. If the orbit of a periodic solution consists of two
 closed curves, then it is
called a {\it double-choreographic solution}. If the orbit of a periodic solution consists of different closed curves, each of which is the trajectory of exact one body,  it is called {\it non-choreographic solution}. Many relative equilibria give rise to simple choreographic solutions and they are called trivial choreographic solutions (circular motions).
The first remarkable non-trivial choreographic solution -- the figure eight of the three body problem  was  discovered numerically by Moor (1993 \cite{Moor}) and it was proved by the variational method by Chenciner and Montgomery (2000, \cite{CM}). Many expertises attempt to study choreographic solutions and a large number of simple choreographic solutions have been discovered numerically  but very few of them have rigorous existence proofs. More results can be found in \cite{ABT, BT2, BCPS, CGMS, CV, Chen2, Ouyang1,DengZhangZhou} and the reference therein.

In 2002, Chenciner-Gever-Montgomery-Sim\'o made a conjecture  in \cite{CGMS} that there exists a nontrivial simple choreographic solution for the equal mass Newtonian $n$-body problem for every $n\geq 3$. Although we focus on 4-body problem in this paper, our method can be used to prove the existence of choreographic solutions in general Newtonian $n$-body problem for every $n\geq 3$. Furthermore, we claim that there exist infinitely many simple choreographic solutions for every $n\geq 4$ and we will study this in future papers. 

 Figure eight is a remarkably  non-trivial simple choreographic solution, but more importantly, it is stable and the stability was proved in (\cite{KS, RG}). Most important to astronomy are stable periodic solutions which means that there is some chance that such periodic solutions might actually be seen in some stellar system. To the best knowledge of the authors, all of the above known simple choreographic solutions are unstable except the figure eight. It seems very hard to find a stable simple choreographic solution (C. Sim\'o \cite{SM} and R. Vanderbei \cite{Va}). Surprisingly, many simple choreographic orbits in this paper  are linearly stable. 
%


 In this paper, a new variational method with structural prescribed boundary conditions (SPBC) which is mainly developed by one of the authors, T. Ouyang, is introduced to construct periodic solutions of the $n$-body problem. The original motivation for our work is practical  and aesthetic: we want to present a concise and effective method not only to find many different types of beautiful new periodic motions for general $n$-body problem but also actually to prove the existence of these motions. Instead of considering the whole path in some loop space with some topological or geometric symmetric constrains, we consider the boundary value problem with appropriate prescribed boundary configurations  which are given by two $n \times d$ matrices, $Qstart$ and $Qend$. For instance in the following theorem \ref{main}, $q(0)=Qstart$ and $q(T)=Qend$ are two $4 \times 2$  matrices possessing an appropriate prescribed structure. The most innovative improvement of this method is to utilize a two-step minimizing process with a proper SPBC to find some appropriate pieces of orbits so that they can be assembled out to a periodic solution (or a quasi-periodic solution). Minimizers are obtained in the full space (not in a restricted symmetric space) with the SPBC.
  The method works for general $n$-body problem without any constraints on masses or symmetries.

Using the proposed variational method with the SPBC, we give rigorous existence proofs of infinitely many simple choreographic solutions. The simulation of the solutions for the n-body problem can be found at http://sest.vsu.edu/$\sim$zxie/N\underline{\hbox{ }}body\underline{\hbox{ }}Simulation.htm. Surprisingly, a family of choreographic orbits of this type are all linearly stable by numerical analysis. Among the many stable simple choreographic orbits,  the most extraordinary one is the stable star pentagon choreographic solution (see Figure \ref{fig1}).  The star pentagon is assembled out of four pieces of curves which are obtained by minimizing the Lagrangian action functional over the SPBC.  We also prove the existence of infinitely many double choreographic periodic solutions, infinitely many non-choreographic periodic solutions and uncountably many quasi-periodic solutions. Each type of periodic solutions have many stable solutions and  possibly  infinitely many stable solutions.

\begin{theorem}\label{main}
For any fixed $T>0$, and $\theta=\frac{2\pi}{5}$, let the structural prescribed boundary conditions be two fixed boundary configurations $q(0)= \left( \begin{array}{cc} a_1 & a_2\\
0 & -a_3\\ -a_1 & a_2\\
0 & -2a_2+a_3 \end{array} \right)$ and $q(T)= \left( \begin{array}{rr} -a_5 & a_4\\
a_5 & a_4\\ -a_6 & -a_4\\
a_6 & -a_4 \end{array} \right) R(\theta)$, where $\vec{a}=(a_1, a_2, \cdots, a_6)\in \mathbf{R}^6$, and the rotation matrix $R(\theta)= \left( \begin{array}{ll} \cos(\theta) & -\sin(\theta)\\
\sin(\theta) & \cos(\theta)\end{array} \right)$. Then there exists an $\vec{a}_0\in \mathbf{R}^6$ such that a minimizing path $q^*(t)=(q_1^*(t), q_2^*(t), q_3^*(t), q_4^*(t))$  on $[0, T]$  connecting $q(0)$ and $q(T)$  can be extended to a periodic solution $q(t)$ of the Newton's equation.
The periodic solution $q(t)$ (see Figure \ref{fig1}) has minimum period $\mathcal{T}=40T$ with the following properties:
\begin{enumerate}
\item[1.] (Noncollision) $q_i(t)\not=q_j(t)$ for any $t$ and $i\not=j$.
\item[2.] (Choreographic) $q_2(t)=q_1(t+10T)$, $q_3(t)=q_1(t+20T)$, $q_4(t)=q_1(t+30T)$, and $q_1(t)=q_1(t+40T)$.
\item[3.] (Symmetry) $q_1(-t)=q_3(t)B$, $q_2(-t)=q_2(t)B$, $q_3(-t)=q_1(t)B$,  and $q_4(-t)=q_4(t)B$, where $B=\left( \begin{array}{ll} -1 & 0\\
0 & 1\end{array} \right)$ is the reflection about $y$-axis, i.e. $q_j(t)B=(-q_{j1}(t), q_{j2}(t))$.
\item[4.] (Geometric Transition) $\{q_i(2kT)\}$ are vertices of an isosceles triangle with an interior point on the axis of the isosceles triangle. $\{q_i((2k+1)T)\}$ are vertices of a trapezoid.
    \item[5.] (Stability) $q(t)$ is a linearly stable star pentagon choreographic solution.
\end{enumerate}
\end{theorem}
\begin{remark}
 Numerically, if $T=1$ and $ \theta=\frac{2\pi}{5}$, then $\vec{a}_0 =[  1.0598738926379,$ $   1.7699901770118,$  $  0.80951135793043,$ $  0.75377929101531,$ $  1.1034410399611,$ $   2.440248251576]$ and  the initial conditions are given in Figure \ref{fig1}. 
 The star pentagon forms by assembling out the initial four pieces of curves starting from $q(0)$ to $q(T)$ (dashdotted line in Figure \ref{fig1}). The extension to the full star pentagon is done by reflection, permutation and rotation as in equation \eqref{qet}.
\end{remark}

Our main theorem for $\theta =\frac{2\pi}{5}$ can also be extended for some other $\theta$. 
\begin{theorem}\label{Main2}
There exist $\theta_0$ and $\theta_1$ such that  $\theta_0 <\frac{\pi}{2}<\theta_1$. For any $\theta\in (\theta_0, \theta_1)$ and $\theta\not=\frac{\pi}{2}$, there exists at least one $\vec{a}_0\in \mathbf{R}^6$ such that the minimizing path $q^*(t)$ connecting $q(0)$ and $q(T)$ can be extended to a non-circular classical Newtonian solution by assembling out the initial pieces. If $\theta$ is not commensurable with $\pi$, the extension is a quasi-periodic solution. If $\theta$ is commensurable with $\pi$, the extension is a periodic solution. Among the periodic solutions, there are infinitely many  choreographic periodic solutions and many of them are linearly stable (see Figure \ref{fig9} to Figure \ref{fig10} in appendix C).
\end{theorem}

\begin{remark}\label{rem2}
 The value of the action of the minimizing solution $q^*(t)$ is smaller than the action of the corresponding circular motion for $\theta\in(\theta_0,\theta_1)$. Numerically, $1.1938<\theta_0<1.2252$ or equivalently $0.38\pi<\theta_0<0.39\pi$. $1.7279<\theta_1<1.7593$ or equivalently $0.55\pi<\theta_1<0.56\pi$. There also exist local minimizers which have higher actions than their circular solutions. 
  By using canonical transformation, we eliminate the trivial $+1$ multipliers and we prove that the simple choreographic solutions  are linearly stable in the reduced space for $\theta=\frac{P}{2P+1}\pi$, $P=2,3,4,\cdots, 15$. The non-choreographic solutions  are linearly stable for $\theta=\frac{2P-1}{4P},$ $P=3,4,\cdots,8$. The double choreographic solutions $q(t)$ are also linearly stable for $\theta=\frac{2P-1}{4P+2},$ $P=5,6, 7$. From our calculation and numerical simulation program, periodic solutions with lower actions for $\theta\in (\theta_0,\theta_1)$ seem more likely stable. Periodic solutions for $\theta$ out of $(\theta_0,\theta_1)$ are more likely unstable. For example, we check that periodic solutions are linearly unstable for $\theta=\frac{1}{4}\pi,$ $\frac{1}{3}\pi,$ $\frac{3}{8}\pi,$ $\frac{3}{10}\pi$, $\frac{4}{11}\pi$,  $\frac{5}{14}\pi$ and so on.  Our theorem \ref{Thm:51} and numerical computation supports the following conjecture. To prove the conjecture, some new techniques may be involved such as index theory (see \cite{HLS}, \cite{YL}, \cite{DO})\\

{\bf \Large Conjecture:} The non-circular periodic solutions in theorem \ref{Main2} are all linearly stable for $\theta\in (\theta_0,\theta_1)$ and $\theta\not=\frac{\pi}{2}$ and there are  infinitely many stable choreographic solutions.
\end{remark}

Our paper is organized in the following manner. In section \ref{sec2}, we first briefly describe the variational method with structural prescribed boundary conditions for the special case $\theta=\frac{2\pi}{5}$ and equal masses. The main theorem \ref{main} is restated as theorem \ref{Thm:NC} to \ref{Thm:LS}. Its proof is carried out in section \ref{sec3} and  the linear stability is studied in section \ref{sec4}. Section \ref{sec5} is devoted to the main theorem \ref{Main2} with more details on the classification of periodic solutions with respect to the general rotational angle $\theta$. The calculations of the action of a path to generate a circular motion or the action of a test path are given in appendix A and appendix B respectively. Finally, some numerical simulations are given in appendix C.

\section{Settings and Restatements of Main Theorem \ref{main}}\label{sec2}

Given $n$ bodies, let $m_i$ denote the mass and $q_i(t)$ denote the position in $\mathbf{R}^d, d\geq 2$ of body $i$ at time $t$ in $d$-dimensional space. The {\it action functional} is a mapping from the space of all trajectories $q_1(t), q_2(t),\cdots, q_n(t)$ into the reals.  It is defined as the integral:
\begin{equation}\label{Min}
\mathcal{A}(q(t))=\int_{0}^{T}  \frac{1}{2}\sum_{i=1}^{n} m_i\|\dot{q}_i(t)\|^2+U(q(t))dt,
\end{equation}
where $U$ is the Newtonian
potential function
$U=\sum_{1\leq i<j\leq n}\frac{m_i m_j}{|{q}_i-{q}_j|}.$
Critical points of the action functional are trajectories that satisfy the equations of motion, i.e. Newton's equations:
\begin{equation}\label{Newton}
m_i\ddot{q}_i=\frac{\partial U}{\partial q_i}=\sum_{j=1,j\not= i}^{n} \frac{m_im_j(q_j-q_i)}{|q_j-q_i|^3} \hspace{1cm} 1\leq i\leq n.
\end{equation}
Without loss of generality, we assume that the center of mass $\bar{q}=(1/M)\sum_{i=1}^{n}m_iq_i$ is always at the origin. Let $p_i=m_i\dot{q}_i$. Then the Hamiltonian governing the equations of motion is $H(q,p)=\sum_{i=1}^{n}|p_i|^2/(2m_i) -U$. \\

By the fundamental theorem of existence and uniqueness of differential equations, the second order nonlinear ODE system of  Newton's equations has a unique solution for an appropriate initial conditions, i.e. the initial position vector $q(0)$ and the initial velocity vector $\dot{q}(0)$ determine its future motion. But the initial conditions are very local and it is very hard to directly determine the initial conditions to lead a periodic solution. Instead of considering the initial value problem, we consider the boundary value problem
with appropriate prescribed boundary configurations on $q(0)$ and $q(T)$ so that the orbit connecting $q(0)$ and $q(T)$ can be extended to a periodic solution. Here we consider two appropriate boundary configurations $Qstart \in (\mathbf{R}^d)^n$ and $Qend \in (\mathbf{R}^d)^n$, and the path space
 \begin{equation}
 \mathcal{P}(Qstart,Qend):=\{q(t)\in H^1([0,T], (\mathbf{R}^d)^n)\quad {\big | }\quad  q(0)=Qstart, q(T)= Qend\}.
 \end{equation}

A natural choice of the path space for the action functional $\mathcal{A}$ defined in \eqref{Min} is the Sobolev space $H^1([0,T],(\mathbf{R}^d)^n)$, in which a critical point $q(t)$ of $\mathcal{A}$ is a classical solution of Newton's equation \eqref{Newton}  on $[0,T]$, if and only if $q$ is collision free. The existence of minimizers in the Sobolev space is classic and standard. But the assertion of collision free for the boundary value problem is proved by Chenciner \cite{CA2} and Marchal \cite{Ma1} in 2002. Such ideas had been also reported by one of the authors, Tiancheng Ouyang, in Guanajuato (Hamsys, March 2001, see Chenciner's Remark in \cite{CA2}).

\begin{lemma}\label{MaC}
 Given any $Qstart \in (\mathbf{R}^d)^n$ and $Qend\in (\mathbf{R}^d)^n$, minimizers of $\mathcal{A}$ on the space $\mathcal{P}(Qstart,Qend)$ are collision-free on the interval $(0,T)$.
\end{lemma}

Minimizers of $\mathcal{A}$ are classic solutions of the Newton's equation \eqref{Newton} and it is $C^2$ in $\mathcal{P}($ $Qstart, $ $Qend)$. Because linear momentum is an integral of motion, it is natural to assume that every path stays inside the configuration space $W$:
\begin{equation}
W:=\left\{ q\in(\mathbf{R}^d)^n \quad {\big | }\quad \sum_{i=1}^n m_i q_i=0\right\}.
\end{equation}
If $Qstart$ and $Qend$ are both in $W$, then the minimizers are always in $W$.

We will present the new variational approach with the SPBC in the discovery of the stable star pentagon choreographic solution of the planar four-body problem with equal masses. The results can be extended for general $\theta$ in section \ref{sec5}. The method can be applied for unequal masses \cite{OuXie}  and in three dimensional space \cite{OuYan}. In section \ref{sec2} to \ref{sec4}, $\theta=\frac{2\pi}{5}$, $m_1=m_2=m_3=m_4=1$ and dimension $d=2$. The essential part of the new method is the choice of an appropriate SPBC in order to get a possible preassigned periodic orbit. 

\begin{quote} {\bf \Large Structural Prescribed Boundary Conditions (SPBC):} \\
Let $\Gamma=\mathbf{R}^6$. The fixed $Qstart$ and the fixed $Qend$ are defined by $Qstart= \left( \begin{array}{cc} a_1 & a_2\\
0 & -a_3\\ -a_1 & a_2\\
0 & -2a_2+a_3 \end{array} \right)$ and $Qend= \left( \begin{array}{rr} -a_5 & a_4\\
a_5 & a_4\\ -a_6 & -a_4\\
a_6 & -a_4 \end{array} \right) R(\theta)$ for a given $\vec{a}=(a_1,a_2,\cdots,a_6)\in \Gamma$. Then the set $S(\vec{a})$ of minimizers is defined by $$S(\vec{a})=\{ q(t) =(q_1, q_2, q_3, q_4)(t) \in C^2((0,T),(\mathbf{R^2})^4)\quad {\big | }\quad   q(0)=Qstart, q(T)=Qend,$$ $$ q(t) \hbox{ is a minimizer of the action functional } \mathcal{A} \hbox{ over } \mathcal{P}(Qstart,Qend) \}.$$
So the configuration of the bodies changes from an isosceles triangle with one on the axis of symmetry of the triangle to a trapezoid for some positive $\vec{a}$.\end{quote} 

 For any given $\vec{a}\in \Gamma$, the minimizers of $\mathcal{A}$ that connect $Qstart$ and $Qend$ are classical collision-free solutions in the interval $(0, T)$. 
 Then the real value function $\tilde{\mathcal{A}}(\vec{a}): \Gamma\rightarrow \mathbf{R}$ is well defined by
 \begin{equation}\label{VAR1}
 \tilde{\mathcal{A}}(\vec{a}) = \int_{0}^{T}  \frac{1}{2}\sum_{i=1}^{n} m_i\|\dot{q}_i(t,\vec{a})\|^2+U(q(t,\vec{a}))dt,
 \end{equation}
where $q(t,\vec{a})\in S(\vec{a})$  is a minimizer of the action functional $\mathcal{A}$ over $\mathcal{P}(Qstart,Qend)$ for the given $\vec{a}\in \Gamma$. If it is clear that $q(t,\vec{a})$ is a minimizer for the given $\vec{a}$ from context, we still use $q(t)$ for $q(t,\vec{a})$ for convenience. It is easy to know that $\tilde{\mathcal{A}}$ is lower semicontinuous on $\Gamma$.
The existence of minimizers in the finite dimension space $\Gamma$ is due to the following proposition.
 \begin{proposition}\label{PROP:1} 
 For $\theta=\frac{2\pi}{5}$, $\tilde{\mathcal{A}}(\vec{a})\rightarrow +\infty$ if $|\vec{a}|\rightarrow +\infty$.
  \end{proposition}
  \begin{proof}
  \begin{figure}
\includegraphics[height=6cm,width=.45\textwidth]{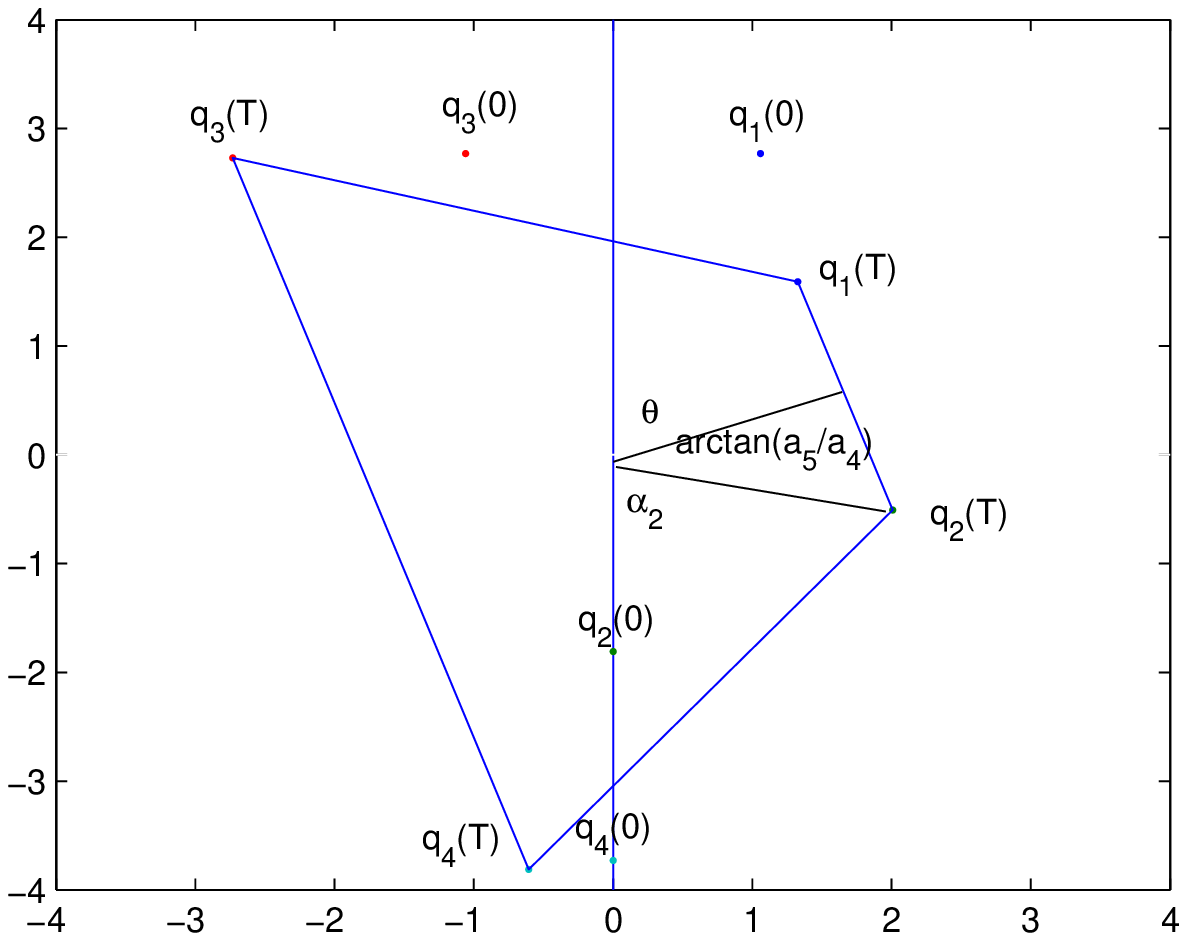}
\includegraphics[height=6cm,width=.45\textwidth]{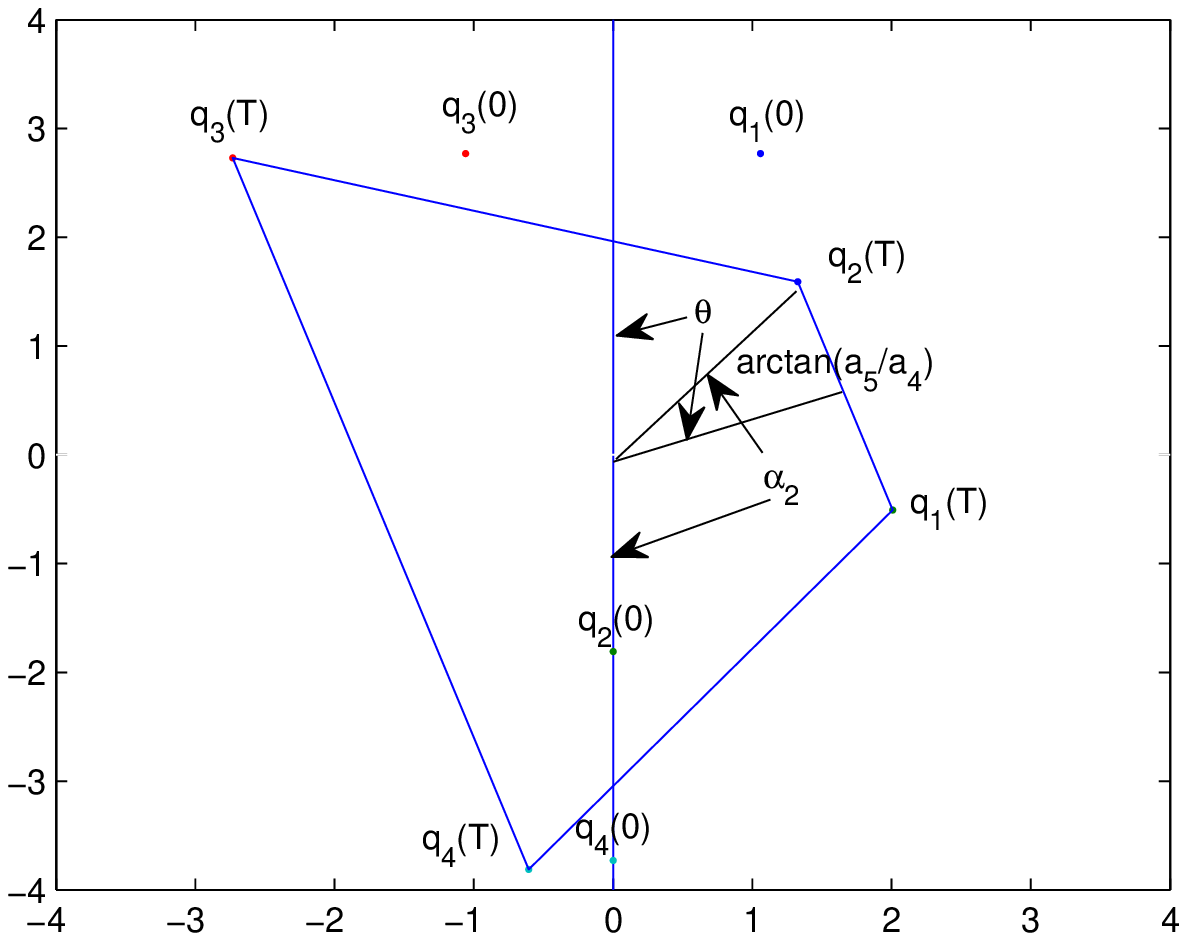}
\caption{ Left: If $\vec{a}> 0$, $0<\alpha_2=\pi-\theta-\arctan(\frac{a_5}{a_4})<\pi$. Right: $a_5<0$ and $a_i>0$ for $i\not=5$, $\frac{\pi}{2}<\alpha_2=\pi-\theta+\arctan(\frac{a_5}{a_4})$.}\label{prop_1}\end{figure}

  For any $\vec{a}\in \Gamma$,
  $$\tilde{\mathcal{A}}(\vec{a}) \geq \sum_{i=1}^{n} \int_{0}^{T}  \frac{1}{2} m_i\|\dot{q}_i(t,\vec{a})\|^2 dt\geq \sum_{i=1}^{n}\frac{1}{2}m_i \left\|\int_{0}^{T} \dot{q}_i(t,\vec{a})dt \right\|^2=\sum_{i=1}^{n}\frac{1}{2}m_i \left\| q_i(T)-q_i(0)\right\|^2.$$ By triangle inequality we have
  \begin{equation}\label{prop1}
\tilde{\mathcal{A}}(\vec{a}) \geq \sum_{i=1}^{n}\frac{1}{2}m_i (\|q_i(T)\|-\|q_i(0)\|)^2
\end{equation}
 \begin{equation}\label{prop2}
 \tilde{\mathcal{A}}(\vec{a}) \geq \sum_{i=1}^{n}\frac{1}{2}m_i |\max\{\|q_i(T)\|\sin(\alpha_i),\|q_i(0)\|\sin(\alpha_i)\}|^2,
 \end{equation}
  \begin{equation}\label{prop3}
 \tilde{\mathcal{A}}(\vec{a}) \geq \frac{1}{2}m_i |\max\{\|q_i(T)\|,\|q_i(0)\|\}|^2, \hbox{ if } \frac{\pi}{2}\leq \alpha_i\leq \pi
 \end{equation}
  where
 $\alpha_i$ is the angle between the vectors $q_i(0)$ and $q_i(T)$.  If $\tilde{\mathcal{A}}(\vec{a})$ remains finite while $|\vec{a}|$ goes to infinity, then $q_i(0)$ and $q_i(T)$ must go to infinity for all $i$ by the structure of the SPBC and the inequality \eqref{prop1}. Then  $\sin(\alpha_i)$ goes to zero for all $i$ by the inequality \eqref{prop2} and \eqref{prop3}, that is $\alpha_i=0$. But that $\alpha_i=0$ for all $i$ is impossible due to the structure of the SPBC. For example, if $a_i\geq 0$ for all $i$, the angle  $\alpha_2$  would be $\pi-\theta-\arctan(\frac{a_5}{a_4})$ (see Left in figure \ref{prop_1}). Therefore $0<\frac{\pi}{2}-\theta\leq \alpha_2\leq \pi-\theta<\pi$ and $\sin(\alpha_2)\not=0$. Other cases can be easily obtained by the geometric structure of the SPBC and detail arguments are omitted.

\end{proof}

Let $\vec{a}_0=(a_{10},a_{20},\cdots,a_{60})\in \Gamma$ be a minimizer of $\tilde{\mathcal{A}}(\vec{a})$ over the space $\Gamma$ and the corresponding path $q^*(t)=q^*(t,\vec{a}_0)\in S(\vec{a}_0)$, i.e.
  \begin{equation}\label{VAR}
  \begin{array}{ll}
  \tilde{\mathcal{A}}(\vec{a}_0) &= \min_{\vec{a}\in \Gamma} \tilde{\mathcal{A}}(\vec{a}) =\min_{\vec{a}\in\Gamma}\left\{\inf_{q(t)\in \mathcal{P}(Q_{start},Q_{end})} \mathcal{A}(q(t))\right\}\\
  \\
  &=\min_{\vec{a}\in\Gamma}\left\{\inf_{q(t)\in \mathcal{P}(Q_{start},Q_{end})}\int_{0}^{T}  \frac{1}{2}\sum_{i=1}^{n} m_i\|\dot{q}_i(t)\|^2+U(q(t))dt\right\}. \end{array}
 \end{equation}
   Then the path $q^*$  is the solution we want and Theorem \ref{main} can be proved immediately by the following theorems.


\begin{theorem}[Noncollision]\label{Thm:NC}
 Let $\vec{a}_0$ be a minimizer of $\tilde{\mathcal{A}}(\vec{a})$ over the space $\Gamma$ and the corresponding path $q^*(t)\in S(\vec{a}_0)$. Then $q^*$ satisfying SPBC  is a classical solution of Newton's equation \eqref{Newton} in the whole interval $[0, T]$.
\end{theorem}

Let $\mathbf{A}$ and $\mathbf{B}$ be two proper linear subspaces of $(\mathbf{R}^2)^4$ which are given as $$\mathbf{A}=\left\{ \left. \left( \begin{array}{cc} a_1 & a_2\\
0 & -a_3\\ -a_1 & a_2\\
0 & -2a_2+a_3 \end{array} \right)\in (\mathbf{R}^2)^4\right| (a_1,a_2,a_3)\in \mathbf{R}^3 \right\}$$ and $$\mathbf{B}=\left\{ \left.  \left( \begin{array}{rr} -a_5 & a_4\\
a_5 & a_4\\ -a_6 & -a_4\\
a_6 & -a_4 \end{array} \right) R(\theta) \in (\mathbf{R}^2)^4\right| (a_4,a_5,a_6)\in \mathbf{R}^3 \right\}.$$
Let us consider the action functional $\mathcal{A}$ defined in \eqref{Min} over the function space $$\mathcal{P}(\mathbf{A},\mathbf{B}):=\{q\in H^1([0,T],(\mathbf{R}^2)^4) | q(0)\in \mathbf{A}, q(T)\in \mathbf{B}\}.$$
It is easy to prove the theorem of equivalence below.
\begin{theorem}[Equivalence]\label{Thm:EQ}
$\vec{a}_0\in \Gamma$ with correspoinding path $q^*\in S(\vec{a}_0)$ satisfying $q^*(0)=Qstart$ and $q^*(T)=Qend$ is a minimizer of $\tilde{\mathcal{A}}(\vec{a})$ over the space $\Gamma$, if and only if, $q^*$ is a minimizer of $\mathcal{A}$ over the function space $\mathcal{P}(\mathbf{A},\mathbf{B})$ with $q^*(0)=Qstart\in \mathbf{A}$ and $q^*(T)=Qend \in \mathbf{B}$.
\end{theorem}



\begin{theorem}[Extension]\label{Thm:ET} 
For any local minimizer $\vec{a}_0\in \Gamma$ of $\tilde{\mathcal{A}}(\vec{a})$ over the space $\Gamma$,  its corresponding path $q^*(t)=(q_1^*(t), q_2^*(t), q_3^*(t), q_4^*(t))$  on $[0, T]$ can be extended to a periodic solution of the Newton's equation \eqref{Newton} by the reflection $B=\left( \begin{array}{ll} -1 & 0\\
0 & 1\end{array} \right)$, a permutation $\sigma$ and the rotation $R(\theta)$ as follows: $q(t)=q^*(t)$ on $[0, T]$,
 $$q(t)=((q^*_2(2T-t),q^*_1(2T-t),q^*_4(2T-t),q^*_3(2T-t))B)R(2\theta) \quad \hbox{ on } \quad (T, 2T],$$
and
 \begin{equation}\label{qet}
 q(t)=\sigma^{k}(q(t-2kT))R(2k\theta) \hbox{ for } t\in (2kT,(2k+2)T]
  \hbox{ and } k\in \mathbf{Z}^+,
 \end{equation}
where $\sigma=[2,3,4,1]$ is a permutation such that  $\sigma(q(t-2T))=( q_2(t-2T), q_3(t-2T), q_4(t-2T), q_1(t-2T)).$ There exists a local minimizer $\vec{a}_0$ such that its corresponding path produces a star pentagon choreographic solution (see Figure \ref{fig1}) of the Newton's equation \eqref{Newton} with  minimum period $40T$.
\end{theorem}

\begin{theorem} \label{Thm:LS}
Star pentagon choreographic solution is linearly stable.
\end{theorem}

\begin{remark}\label{remark2}
(1) These theorems assert  that the initial pieces of orbits are extended to a periodic solution by assembling the pieces themselves. From the extension equation \eqref{qet}, it is easy to prove the properties in theorem \ref{main}.

(2) This idea had been reported by one of the authors, Tiancheng Ouyang \cite{CA2} in 2001. Some preliminary results were included in the paper \cite{Ouyang1} which was submitted in 2003 and was finally published in 2012.
This method can also be used to prove the existence of figure eight, spatial isosceles three-body problem, Hip-Hop and the other surprising motions that have been recently discovered. For example, the SPBC for figure eight can be chosen as $Qstart= \left( \begin{array}{cc} a_1 & -a_2\\
-a_1& a_2\\ 0 & 0 \end{array} \right)$ and $Qend= \left( \begin{array}{rr} a_3 & 0\\
-\frac{a_3}{2} & -a_4\\ -\frac{a_3}{2} & a_4\end{array} \right).$ 

(3) Proposition  \ref{PROP:1}  is generally true for $\theta\in (0, \pi)\backslash\left\{\frac{\pi}{4},\frac{\pi}{2},\frac{3\pi}{4}\right\}$. It can be proved by simply noticing that $\mathbf{A}\bigcap \mathbf{B}$ is  empty for $\theta\in (0, \pi)\backslash\left\{\frac{\pi}{4},\frac{\pi}{2},\frac{3\pi}{4}\right\}$.  When  $\theta\in \left\{\frac{\pi}{4},\frac{\pi}{2},\frac{3\pi}{4}\right\}$, there exists a sequence $\vec{a}_n$ such that $\tilde{\mathcal{A}}(\vec{a_n})$ remains finite while $\|\vec{a}_n\|\rightarrow \infty$ when $n\rightarrow \infty$.
\end{remark}

\section{Proof of Theorem \ref{Thm:NC} to \ref{Thm:ET}}\label{sec3}

\begin{proof}[The proof of theorem \ref{Thm:NC}]
If $\vec{a}_0$ is a minimizer of $\tilde{\mathcal{A}}(\vec{a})$ over the space $\Gamma$, we only need to prove that $Qstart(a_{10},a_{20},a_{30})$ and $Qend(a_{40},a_{50},a_{60})$ have no collision. In fact, there are six cases corresponding to initial collision boundary. (1) $a_{10}\not=0$ and $a_{20}=a_{30}$ binary collision ($m_2$ and $m_4$ collide). (2) $a_{10}=0$, $a_{20}\not=-a_{30}$,and $a_{20}\not=\frac{1}{3}a_{30}$, binary collision ($m_1$ and $m_3$ collide). (3) $a_{10}=0,$ and $a_{20}=a_{30}\not=0$ simultaneous binary collision. (4) $a_{10}=a_{20}=a_{30}=0$ total collision. (5) $a_{10}=0,a_{20}=-a_{30}\not=0$ triple collision ($m_1$, $m_2$, and $m_3$ collide). (6) $a_{10}=0,a_{20}=\frac{1}{3}a_{30}\not=0$ triple collision ($m_1$, $m_3$, and $m_4$ collide).  There are five cases corresponding to ending collision boundary. (7) $a_{40}\not=0$, $a_{50}=0$ and $a_{60}\not=0$ binary collision ($m_1$ and $m_2$ collide). (8) $a_{40}\not=0$, $a_{50}\not=0$ and $a_{60}=0$ binary collision ($m_3$ and $m_4$ collide). (9) $a_{40}\not=0$, $a_{50}=a_{60}=0$ simultaneous binary collision. (10) $a_{40}=a_{50}=a_{60}=0$  total collision. (11) $a_{40}=0,$ $a_{50}=a_{60}\not=0$ simultaneous binary collision.

We only prove the case (1) $a_{10}\not=0$, $a_{20}=a_{30}$ binary collision ($m_2$ and $m_4$ collide) and other cases can be proved by similar arguments. Suppose that $q$ is a local minimizer of $\mathcal{A}$ satisfying the SPBC for $\vec{a}_0$. At time $t=0$ the bodies $m_2$ and $m_4$ start at the collision point $(0,-a_{30})$ while the other bodies are away. Since $q$ has no collision in the open interval $(0,T)$, we will then analyze the motion during the closed time interval $[0,\epsilon]$ and prove the existence of sufficiently small values of $\epsilon$ such that a local deformation has lower action and satisfy the SPBC. The contradiction proves that $q$ can not have this binary collision.  Our proof follows the papers of Marchal \cite{Ma1} and Chenciner \cite{CA2},  but a few technique arguments are proved differently (especially the construction of deformation with SPBC).

We will build the two following solutions $S_2$ (Kepler ejection orbits at the starting point) and $S_3$ (no collision) with: (A) Exactly the same motion of all bodies in the interval $[\epsilon, T)$. (B) At the time interval $[0,\epsilon]$, the ejection orbits are replaced by a collision free orbits with boundary conditions satisfying SPBC. The corresponding actions will be $A_1=\mathcal{A}(q)$, $A_2=\mathcal{A}(S_2)$, $A_3=\mathcal{A}(S_3)$. We want to prove that $A_1>A_3$ for sufficiently small time $\epsilon$. Since (A), the actions are different only in the time interval $[0,\epsilon]$.\\
First, consider the ejection orbits in the starting time interval $[0, \epsilon]$ in $S_2$. Let $r$ be the simple radial two-body motion leading from $0$ to $r_\epsilon$ in the time interval $[0,\epsilon]$. By Sundman and Sperling's estimates near collisions \cite{SH,SK}, there exists a positive constant $\gamma$ such that  $r(t)=(\gamma t^{\frac{2}{3}})\vec{\alpha}$ where $\vec{\alpha}$ is a unit vector.  Let $\xi(t)=\frac{m_2q_2(t)+m_4q_4(t)}{m_2+m_4}$ be the center of mass of the second and forth body.
$$ q_{1S_2}(t)=q_1(t), q_{2S_2}(t)=\xi(t)+\frac{m_4}{m_2+m_4}r(t),$$
$$ q_{3S_2}(t)=q_3(t), q_{4S_2}(t)=\xi(t)-\frac{m_2}{m_2+m_4}r(t).$$
We consider the deformation of $r(t)$ as \begin{equation}
r_{\delta}(t)=r(t)+\delta \phi(t)\vec{s},
\end{equation}
where $\vec{s}$ is a unit vector of $(0,1)$ or $(0,-1)$,  $\delta =\frac{\epsilon}{N}$ with $N\geq 2\max\{ K_{in}/U_{in}, 4\}$, and
$$\phi(t)=\left\{ \begin{array}{ll} 1, & 0\leq t\leq \delta\\
\frac{\delta+\tilde{N}\delta-t}{\tilde{N}\delta}, & \delta<t\leq \delta+\tilde{N}\delta\\
0, & \delta+\tilde{N}\delta<t\leq \epsilon.\end{array}\right.$$
where $K_{in}/U_{in}<\tilde{N}<N-1$. The positive $K_{in}$ and $U_{in}$ are given in the equations \eqref{Kin} and \eqref{Uin} respectively, which are independent of $\epsilon$.

The collision-free motion $S_3$ is denoted by
$$ q_{1S_3}(t)=q_1(t), q_{2S_3}(t)=\xi(t)+\frac{m_4}{m_2+m_4}r_\delta(t),$$
$$ q_{3S_3}(t)=q_3(t), q_{4S_3}(t)=\xi(t)-\frac{m_2}{m_2+m_4}r_\delta(t).$$
At $t=0$, $q_{1S_3}(0)=q_1(0), q_{2S_3}(0)=(0,-a_{30}\pm \delta/2),q_{3S_3}(0)=q_3(0),q_{4S_3}(0)=(0,-a_{30}\mp \delta/2)$. So the initial condition of $S_3$ satisfies the SPBC. \\
Now consider the expression of the actions for each path in the time interval $[0,\epsilon]$. They will be decomposed into two parts: the first part $A_{in}$ is to compute the action of the relative motion of the colliding bodies  $m_2$ and $m_4$; the second part $A_{out}$ is to compute the action of the remainder.

It is easy to know that $A_{1in}\geq A_{2in}$   since the homothetic collision-ejection orbit is a minimizer. We only need to prove $A_{2in}-A_{3in}>A_{3out}-A_{1out}$   in order to prove $A_{1}>A_{3}$ in $[0, \epsilon]$.
 We first note that
$$m_2|\dot{q}_{2S_2}|^2+m_4|\dot{q}_{4S_2}|^2=m_2\langle \dot{\xi}+\frac{m_4}{m_2+m_4}\dot{r},\dot{\xi}+\frac{m_4}{m_2+m_4}\dot{r} \rangle +m_4\langle \dot{\xi}-\frac{m_2}{m_2+m_4}\dot{r},\dot{\xi}-\frac{m_2}{m_2+m_4}\dot{r} \rangle
$$
$$=(m_2+m_4)|\dot{\xi}|^2+\frac{m_2m_4}{m_2+m_4}|\dot{r}|^2.$$
Then
$$A_{2in}-A_{3in}=\int_{0}^{\epsilon}\frac{m_2m_4}{2(m_2+m_4)}(|\dot{r}|^2-|\dot{r}_\delta|^2 ) +m_2m_4\left(\frac{1}{|r|}-\frac{1}{|r_\delta|}\right)dt,$$
$$A_{3out}-A_{1out}=\int_{0}^{\epsilon} \sum_{i=1,3; j=2,4}\left(\frac{m_im_j}{|q_i-q_{j}|}-\frac{m_im_j}{|q_i-q_{jS_3}|}\right)dt.
$$
Now we estimate the bounds for $A_{out}$. Consider the motion of the mass $m_j$ between the arbitrary successive instants $t_1$ and $t_2$. Because the minimum of the integral $\int_{t_1}^{t_2} \frac{m_j |\dot{q}_j|^2}{2}dt$ between given positions $q_j(t_1)$ and $q_j(t_2)$ is obtained for a constant velocity vector, we can always write $\frac{m_j|q_j(t_2)-q_j(t_1)|^2}{2(t_2-t_1)}\leq \int_{t_1}^{t_2} \frac{m_j |\dot{q}_j|^2}{2}dt \leq \mathcal{A}(q)\leq K< \infty$. So if $0\leq t_1\leq t_2\leq T$, $|q_j(t_2)-q_j(t_1)|\leq \left(\frac{2K(t_2-t_1)}{m_j}\right)^{1/2}$. Pick up $\epsilon >0$ small such that the two bodies $m_2$ and $m_4$ will remain at less than twice that distance from the collision point $(0,-a_{30})$ all along the time interval $[0, \epsilon]$, i.e. $|q_2-q_4|\leq J\sqrt{\epsilon}$, where $J=2(2K)^{1/2}$. $m_1$ and $m_3$ will remain outside of the circle centered at the collision point with radius $D$ and $J\sqrt{\epsilon}\leq J\sqrt{\epsilon_0}\ll D$ for a fixed $\epsilon_0$. So during the time interval $[0, \epsilon]$, the two bodies $m_1$ and $m_3$ are outside of the circle of radius $D$ and center $(0,-a_{30})$, while the bodies $m_2$ and $m_4$ are inside the much smaller circle of the same center and radius $J\sqrt{\epsilon}$.
$$|A_{3out}-A_{1out}|\leq \int_{0}^{\epsilon} \sum_{i=1,3; j=2,4}m_im_j\left|\left(\frac{|q_i-q_{jS_3}|-|q_i-q_{j }| }{|q_i-q_{j }| |q_i-q_{jS_3}|} \right)\right|dt.
$$
$$  \leq \int_{0}^{\epsilon} \sum_{i=1,3; j=2,4}m_im_j\left(\frac{|q_{jS_3}-q_{j }| }{|q_i-q_{j }| |q_i-q_{jS_3}|} \right)dt.
$$
\begin{equation}\label{Uout} \leq \int_{0}^{\epsilon} \sum_{i=1,3; j=2,4}m_im_j\left(\frac{J\sqrt{\epsilon} }{(D-J\sqrt{\epsilon_0})^2} \right)dt= \frac{4J}{(D-J\sqrt{\epsilon_0})^2}\epsilon^{\frac{3}{2}}=U_{out}\epsilon^{\frac{3}{2}}.
\end{equation}

Let us compute $A_{2in}-A_{3in}$. By choosing appropriate $\vec{s}$
 such that  $\langle r,\vec{s}\rangle\geq 0$,
$$\int_{0}^{\epsilon}\frac{m_2m_4}{2(m_2+m_4)}(|\dot{r}|^2-|\dot{r}_\delta|^2 )dt=-\int_{0}^{\epsilon}\frac{m_2m_4}{2(m_2+m_4)}(2\delta\dot{\phi}\langle r,\vec{s}\rangle +(\delta\dot{\phi})^2)dt$$
$$\geq -\int_{\delta}^{\delta+\tilde{N}\delta}\frac{m_2m_4}{2(m_2+m_4)} (\delta\dot{\phi})^2dt =-\int_{\delta}^{\delta+\tilde{N}\delta}\frac{m_2m_4}{2(m_2+m_4)}( -\frac{1}{\tilde{N}})^2dt$$
\begin{equation}\label{Kin}
\geq-\frac{m_2m_4}{2(m_2+m_4)}\frac{\delta}{\tilde{N}}=-K_{in}\frac{\delta}{\tilde{N}} .
\end{equation}

$$\int_{0}^{\epsilon}\left(\frac{1}{|r|}-\frac{1}{|r_\delta|}\right)dt= \int_{0}^{\epsilon}\left(\frac{1}{|r|}-\frac{1}{(|r|^2+2\delta\phi\langle r,\vec{s}\rangle+(\delta \phi)^2)^{1/2}}\right)dt
$$
$$=\int_{0}^{\epsilon}\left(\frac{2\delta\phi\langle r,\vec{s}\rangle+(\delta \phi)^2}{|r|(|r|^2+2\delta\phi\langle r,\vec{s}\rangle+(\delta \phi)^2)^{1/2} (|r|+ (|r|^2+2\delta\phi\langle r,\vec{s}\rangle+(\delta \phi)^2)^{1/2})}\right)dt
$$
$$\geq  \int_{0}^{\delta}\left(\frac{(\delta )^2}{|r|(|r|+\delta) (2|r|+ \delta )}\right)dt \geq  \int_{0}^{\delta}\left(\frac{1}{\gamma(\gamma+\delta^{1/3}) (2\gamma+ \delta^{1/3} )}\right)dt
$$
\begin{equation}\label{Uin}
\geq \left(\frac{1}{\gamma(\gamma+1) (2\gamma+ 1)}\right)\delta=U_{in}\delta
 \end{equation}
So $A_{2in}-A_{3in}>\left(-K_{in}\frac{\delta}{\tilde{N}}+ U_{in}\delta\right)=\left(-\frac{K_{in}}{\tilde{N}}+U_{in}\right)\frac{\epsilon}{N}>U_{out}\epsilon^{\frac{3}{2}}\geq A_{3out}-A_{1out}$
for small $\epsilon$, which implies $A_1>A_3$. \\
The path $S_3$ starts at $q_1=(a_{10},a_{20}),$ $q_2=(0,-a_{20}\pm\delta/2)$, $q_3=(-a_{10},a_{20})$, $q_4=(0,-a_{20}\mp\delta/2)$ which satisfy the SPBC. The action of $S_3$ is smaller than the action of $S_1$ which contradicts the fact that $S_1$ is a minimizer. \\
The contradiction completes the proof that the vector $\vec{a}_0$ with collision $a_{20}=a_{30}$ is not a minimizer of $\mathcal{A}$ on $\Gamma$. \\
By the structure of the SPBC, we can prove that any corresponding path $q(t)$ of a minimizer of $\tilde{\mathcal{A}}$ on $\Gamma$ has no collision in $[0,T]$ by similar arguments for other cases. The feature of the SPBC that we apply to the arguments of  non-collision for other possible collision boundary cases is that SPBC has enough number of free variables. For example, for the case (10) $a_{40}=a_{50}=a_{60}=0$ total collision, the collision boundary can be deformed to a central configuration by perturbing the values of variables $a_{40},a_{50}, a_{60}$. The total collision solution can be replaced locally by a homographic solution generated by a central configuration.

\end{proof}

\begin{proof}[The proof of theorem \ref{Thm:ET}]
 By theorem \ref{Thm:NC}, any path $q^*(t)$ corresponding to a local minimizer $\vec{a}_0$ is a classic solution in the interval $[0, T]$. We prove this theorem by two steps. First, we prove that it can be extended to a periodic solutions. Second, we prove that there exists a local minimizer $\vec{a}_0$ which produce the star pentagon solution. \\
Because $q^*(t)$ is a classic solution of Newton's equation \eqref{Newton} on $(0,T)$, it is easy to check that $q(t)$ is a classical solution in each interval $((n-1)T, nT)$ for any given positive integer $n$. To prove $q(t)$ is a classical solution for all real $t$, we need prove that $q(t)$ is connected very well at $t=nT$ for any integer $n$, i.e. $\lim_{t\rightarrow (nT)^-} q(t)=\lim_{t\rightarrow (nT)^+} q(t)$ and $\lim_{t\rightarrow (nT)^-} \dot{q}(t)=\lim_{t\rightarrow (nT)^+} \dot{q}(t)$. By the structure of  the extension equation \eqref{qet}, we only need prove it for $n=1$ and $n=2$. By the SPBC, we have $\lim_{t\rightarrow (nT)^-} q(t)=\lim_{t\rightarrow (nT)^+} q(t)$ at $n=1$ and $n=2$. That $\lim_{t\rightarrow (nT)^-} \dot{q}(t)=\lim_{t\rightarrow (nT)^+} \dot{q}(t)$ at $n=1$ and $n=2$ is equivalent to the relations given by \eqref{TT} and \eqref{T0} below.
 \begin{equation}\label{TT}
 \begin{array}{ll}
\dot{q}_{1}(T)=(\dot{q}_{21}(T), -\dot{q}_{22}(T))R(2\theta), &
\dot{q}_{2}(T)=(\dot{q}_{11}(T), -\dot{q}_{12}(T))R(2\theta),\\
\dot{q}_{3}(T)=(\dot{q}_{41}(T), -\dot{q}_{42}(T))R(2\theta),&
\dot{q}_{4}(T)=(\dot{q}_{31}(T), -\dot{q}_{32}(T))R(2\theta),
\end{array}
\end{equation}
and at
$t=2T$, \begin{equation}\label{T0}
\dot{q}_{11}(0)=\dot{q}_{31}(0), \dot{q}_{12}(0)=-\dot{q}_{32}(0), \dot{q}_{22}(0)=\dot{q}_{42}(0)=0.
\end{equation}

 Since $\vec{a}_0\in \Gamma$ is a minimizer of $ \tilde{\mathcal{A}}(\vec{a})$ over $\Gamma$, $q^*$ is a minimizer of $\mathcal{A}$ over the function space $\mathcal{P}(\mathbf{A},\mathbf{B})$ by theorem \ref{Thm:EQ}. Here we use $q$ for $q^*$ by our extension formula \eqref{qet}. Consider an admissible variation  $\xi\in \mathcal{P}(\mathbf{A},\mathbf{B})$ with $\xi(0)\in \mathbf{A}$ and $\xi(T)\in \mathbf{B}$, then the first variation $\delta_\xi\mathcal{A}(q)$ is computed as:
  $$\delta_\xi\mathcal{A}(q)=\lim_{\delta\rightarrow 0}\frac{\mathcal{A}(q+\delta \xi)-\mathcal{A}(q)}{\delta}$$
    $$=\int_{0}^{T}\frac{1}{2}\sum_{i=1}^{4}\lim_{\delta\rightarrow 0} m_i \frac{\|\dot{q}_i +\delta \dot{\xi}_i\|^2-\|\dot{q}_i \|^2}{\delta}+ \lim_{\delta\rightarrow 0}\frac{U(q+\delta \xi)-U(q)}{\delta}dt$$
$$=\int_{0}^{T} \left (\sum_{i=1}^{4} m_i <\dot{q}_i, \dot{\xi}_i>+ \sum_{i=1}^4<\frac{\partial }{\partial q_i}(U(q(t))),\xi_i>\right)dt$$
$$=\sum_{i=1}^{4}\left.m_i <\dot{q}_i, \xi_i>\right|_{t=0}^{t=T} + \int_{0}^{T} <-m_i \ddot{q}_i+\frac{\partial }{\partial q_i}(U(q(t))),\xi_i>dt. $$
Because the first variation $\delta_\xi\mathcal{A}(q)$ valishes for any $\xi$, $q$ satisfies Newton's equation \eqref{Newton} and   $m_i=1, i=1,2,3,4$, we have
\begin{equation}\label{vaeq}
\delta_\xi\mathcal{A}(q)=\sum_{i=1}^{4}\left(<\dot{q}_i(T),\xi(T)> \right)-\sum_{i=1}^{4}\left(<\dot{q}_i(0),\xi(0)>\right)=0
\end{equation}
For $ i=1,2,3$, let $\xi^{(i)}(t)\in \mathcal{P}(\mathbf{A},\mathbf{B})$ satisfy $\xi^{(i)}(T)=0$ and
$\xi^{(i)}(0)=\left( \begin{array}{cc} a_1 & a_2\\
0 & -a_3\\ -a_1 & a_2\\
0 & -2a_2+a_3 \end{array} \right),$  where $a_i=1,$ $ a_j=0$  if $j\not=i.$ Then

   $$ \delta_{\xi^{(1)}}\mathcal{A}(q)=-(\dot{q}_{11}(0)-\dot{q}_{31}(0))=0,$$
$$ \delta_{\xi^{(2)}}\mathcal{A}(q)=-(\dot{q}_{12}(0)+\dot{q}_{32}(0)-2\dot{q}_{42}(0))=0,$$
$$ \delta_{\xi^{(3)}}\mathcal{A}(q)=-(-\dot{q}_{22}(0)+\dot{q}_{42}(0))=0.$$
By using $\sum_{i=1}^4 \dot{q}_{i1}(0)=0$ and $\sum_{i=1}^4 \dot{q}_{i2}(0)=0$, we can prove that relation \eqref{T0} holds after simple calculation.

For $i=4,5,6,$ let $\xi^{(i)}(t)\in \mathcal{P}(\mathbf{A},\mathbf{B})$ satisfy $\xi^{(i)}(0)=0$ and
$\xi^{(i)}(T)=\left( \begin{array}{cc} -a_5 & a_4\\
a_5 & a_4\\ -a_6 & -a_4\\
a_6 & -a_4 \end{array} \right)R(\theta), $  where $a_i=1,$ $ a_j=0$  if $j\not=i.$ Then
$$ \delta_{\xi^{(4)}}\mathcal{A}(q)=<\dot{q}_1(T)+\dot{q}_2(T),(0,1)R(\theta)>
+<\dot{q}_3(T)+\dot{q}_4(T),(0,-1)R(\theta)>$$
\begin{equation}\label{pa4}
=(\dot{q}_{11}+\dot{q}_{21})\sin(\theta)+(\dot{q}_{12}+\dot{q}_{22})\cos(\theta)
-(\dot{q}_{31}+\dot{q}_{41})\sin(\theta)-(\dot{q}_{32}+\dot{q}_{42})\cos(\theta)=0;
\end{equation}
$$ \delta_{\xi^{(5)}}\mathcal{A}(q)=<\dot{q}_{1}(T), (-1,0)R(\theta)>+<\dot{q}_2(T), (1,0)R(\theta)>$$
\begin{equation}\label{pa5}
=-\dot{q}_{11}\cos(\theta)+\dot{q}_{12}\sin{\theta}+\dot{q}_{21}\cos(\theta)-\dot{q}_{22}\sin(\theta)=0;
\end{equation}
$$\delta_{\xi^{(6)}}\mathcal{A}(q)=<\dot{q}_{3}(T), (-1,0)R(\theta)>+<\dot{q}_4(T), (1,0)R(\theta)>$$
\begin{equation}\label{pa6}
=-\dot{q}_{31}\cos(\theta)+\dot{q}_{32}\sin{\theta}+\dot{q}_{41}\cos(\theta)-\dot{q}_{42}\sin(\theta)=0,
\end{equation}
where the derivatives are taken at $t=T$. \\
Let $$A_{i1}=\dot{q}_{i1}-\dot{q}_{(i+1)1}\cos(2\theta)+\dot{q}_{(i+1)2}\sin(2\theta),$$
$$A_{i2}=\dot{q}_{i2}+\dot{q}_{(i+1)1}\sin(2\theta)+\dot{q}_{(i+1)2}\cos(2\theta),$$
$$A_{j1}=\dot{q}_{j1}-\dot{q}_{(j-1)1}\cos(2\theta)+\dot{q}_{(j-1)2}\sin(2\theta),$$
$$A_{j2}=\dot{q}_{j2}+\dot{q}_{(j-1)1}\sin(2\theta)+\dot{q}_{(j-1)2}\cos(2\theta),$$
for $i=1,3$ and $j=2, 4$. Because $\sum_{i=1}^4 \dot{q}_{ik}=0$ for $k=1,2$, we have
\begin{equation}\label{dqt}
A_{11}+A_{21}+A_{31}+A_{41}=0, \hspace{1cm} A_{12}+A_{22}+A_{32}+A_{42}=0.
\end{equation}
By using the trigonometric identities $\cos(\theta)=\cos(2\theta)\cos(\theta)+\sin(2\theta)\sin(\theta)$ and $\sin(\theta)=\sin(2\theta)\cos(\theta)-\cos(2\theta)\sin(\theta)$, from equation \eqref{pa4}, we have
$$ \dot{q}_{11}\sin(\theta)+\dot{q}_{21}(\sin(2\theta)\cos(\theta)-\cos(2\theta)\sin(\theta)) +\dot{q}_{12}\cos(\theta)+\dot{q}_{22}(\cos(2\theta)\cos(\theta)+\sin(2\theta)\sin(\theta)) -$$ $$\dot{q}_{31}\sin(\theta)-\dot{q}_{41}(\sin(2\theta)\cos(\theta)-\cos(2\theta)\sin(\theta)) -\dot{q}_{32}\cos(\theta)-\dot{q}_{42}(\cos(2\theta)\cos(\theta)+\sin(2\theta)\sin(\theta))=0,$$
which is
\begin{equation}\label{pa41}
A_{11}\sin(\theta)+A_{12}\cos{\theta}-A_{31}\sin(\theta)-A_{32}\cos(\theta)=0,
\end{equation}
Similarly from equation \eqref{pa4}, we also have
\begin{equation}\label{pa42}
A_{21}\sin(\theta)+A_{22}\cos{\theta}-A_{41}\sin(\theta)-A_{42}\cos(\theta)=0,
\end{equation}
\begin{equation}\label{pa44}
A_{21}\sin(\theta)+A_{22}\cos{\theta}-A_{31}\sin(\theta)-A_{32}\cos(\theta)=0.
\end{equation}
 From equation \eqref{pa5} and \eqref{pa6} we have
\begin{equation}\label{pa51}
A_{11}\cos(\theta)-A_{12}\sin{\theta}=0,
\end{equation}
\begin{equation}\label{pa52}
A_{21}\cos(\theta)-A_{22}\sin{\theta}=0,
\end{equation}
\begin{equation}\label{pa61}
A_{31}\cos(\theta)-A_{32}\sin{\theta}=0,
\end{equation}
\begin{equation}\label{pa62}
A_{41}\cos(\theta)-A_{42}\sin{\theta}=0.
\end{equation}
The equations \eqref{pa41}, \eqref{pa51}, \eqref{pa61} imply
\begin{equation}\label{a11}
A_{11}=A_{31}, A_{12}=A_{32}.
\end{equation}
The equations \eqref{pa42}, \eqref{pa52}, \eqref{pa62} imply
\begin{equation}
A_{21}=A_{41}, A_{22}=A_{42}.
\end{equation}
The equations \eqref{pa44}, \eqref{pa51}, \eqref{pa52} and \eqref{a11} imply
\begin{equation}
A_{11}=A_{21}, A_{12}=A_{22}.
\end{equation}
 Then the above three equations and equation \eqref{dqt} imply that $A_{kj}=0$ for $k=1,2,3,4$ and $j=1,2$. Because the relations \eqref{TT} is equivalent to $A_{kj}=0$, we complete the proof that $q(t)$ connects very well at $t=2T$. \\

Now we prove that there exists a minimizing path which is different from the circular motion. The circular motion can be obtained by extending the corresponding minimizing path of a particular local minimizer $\vec{a}^\circ$ in $\Gamma$. Let $q^{\circ}(t)$  be the corresponding path of $\vec{a}^{\circ}$ on $[0, T]$ which can be extended to a circular solution. It is not hard to get the exact formula for action $\mathcal{A}(q^{\circ}(t))$ (see equation \eqref{actCir} in Appendix A):
\begin{equation}
\mathcal{A}(q^{\circ}(t))=3(2)^{-\frac{1}{3}}U_0^{\frac{2}{3}}{T}^{\frac{1}{3}}\left( \theta-\frac{\pi}{4}  \right)^{\frac{2}{3}},\end{equation}
where $U_0=2\sqrt{2}+1$ is the constant value of the potential function for the four-body problem with equal masses 1 at the vertex of unit square. 
For $\theta=\frac{2\pi}{5}$ and $T=1$,  $$\mathcal{A}(q^{\circ}(t))\approx 3.528734094.$$
We assume that the test path $\bar{q}(t)$ is formed by  connecting the straight line with constant velocity from the starting configuration $Qstart$ to the ending configuration $Qend$ for the given $\vec{a}$. We now evaluate the action $\mathcal{A}(\bar{q}(t))$ of the test path $\bar{q}(t)$ for $T=1$ and $\vec{a} =[   1.0597,$ $    1.7696,$  $  0.8094,$ $  0.7536,$ $    1.1032,$ $    2.4398]$. It is not hard to compute the action over the time $[0,T]$. Although the corresponding action can be calculated by hand, it is computed by a Matlab program (see Appendix B).
 $$\mathcal{A}(\bar{q}(t))=3.2484<\mathcal{A}(q^{\circ}(t)).$$

So there exists a local minimizer $\vec{a}_0$ such that $\tilde{\mathcal{A}}(\vec{a}_0)<\mathcal{A}(\bar{q}(t))<\mathcal{A}(q^{\circ}(t))= \tilde{\mathcal{A}}(\vec{a}^\circ)$. Then the corresponding minimizing path $q^*$ of $\vec{a}_0$ produces the star pentagon solution. From the extension equation \eqref{qet}, it is easy to prove other properties of the main theorem \ref{main}. Because $\sigma=[2,3,4,1]$ is the permutation with $\sigma^4(q(t))=q(t)$ and $\theta=\frac{2\pi}{5}$ and the least common multiple of 4 and 5 is 20, the minimum period of the solution  is $\mathcal{T}=40T$ by extension equation \eqref{qet}. When $k=5$, $q(t)=\sigma^{5}(q(t-10T))R(10\theta)=\sigma(q(t-10T))$ for $t\in (10T,12T]$ which implies that the solution is choreographic.
\end{proof}

\section{Linearly Stability of Star Pentagon}\label{sec4}
Suppose that $\gamma(t)$ is a $\mathcal{T}$-periodic solution to the Hamiltonian system $\dot{\gamma}=J\nabla H(\gamma)$, where $J=\left[\begin{array}{ll} 0 & I\\ -I & 0 \end{array}\right]$ is the standard symplectic matrix and $I$ is the appropriately sized identity matrix.  Let $X(t)$ be the fundamental matrix solution to
$$\dot{\xi}=JD^2H(\gamma(t))\xi, \hspace{0.2cm} \xi(0)=I.$$
$X(t)$ is symplectic and satisfies $X(t+\mathcal{T})=X(t)X(\mathcal{T})$ for all $t$. The matrix $X(\mathcal{T})$ is called the monodromy matrix whose eigenvalues, the characteristic multipliers, determine the linear stability of the periodic solution. Since every integral in the $n$-body problem yields a multiplier of $+1$, there are eight $+1$ multipliers for a periodic orbit in the planar problem. It is natural to define the linear stability of a periodic solution by examining stabiltiy on the reduced quotient space.
\begin{definition}
A periodic solution of the planar $n$-body problem has eight trivial
characteristic multipliers of $+1$. The solution is spectrally stable if the remaining
multipliers lie on the unit circle and linearly stable if, in addition, the monodromy matrix
$X(T)$ restricted to the reduced space is diagonalizable.
\end{definition}

Here we apply standard symplectic transform to reduce Hamiltonian system to a 10 dimension Hamiltonian system. The monodromy matrix of the periodic solution $\gamma(t)$ in the reduced system has a pair of $+1$ eigenvalues and the remaining eight eigvalues must be on the unit circle if the solution is linearly stable.

To eliminate the trivial $+1$ multipliers of a periodic solution, we use Jacobi coordinates and symplectic polar coordinates (see chapter 7 in \cite{MHO2}). Denote $p_i=m_i\dot{q}_i$ as the momentum coordinates and let $\mu_i=\sum_{j=1}^i m_j$ and $M_i=\frac{m_i\mu_{i-1}}{\mu_i}$. Then let
$$\begin{array}{ll}
g_4={\frac {{ m_4}\,{ q_4}+{ m_3}\,{ q_3}+{ m_2}\,{ q_2}+{
m_1}\,{ q_1}}{{ m_1}+{ m_2}+{ m_3}+{ m_4}}}, & G_4={ p_4}+{ p_3}+{ p_2}+{ p_1};\\
u_2=q_2 - q_1, & v_2={\frac {{ \mu_1}\,{ p_2}}{{ \mu_2}}}-{\frac {{ m_2}\,{ p_1}}{{
 \mu_2}}};\\
u_3={ q_3}-{\frac {{ m_2}\,{ q_2}+{ m_1}\,{ q_1}}{{ m_1}+{
m_2}}}, & v_3={\frac {{ \mu_2}\,{ p_3}}{{ \mu_3}}}-{\frac {{ m_3}\, \left( {
 p_2}+{ p_1} \right) }{{ \mu_3}}};\\
u_4={ q_4}-{\frac {{ m_3}\,{ q_3}+{ m_2}\,{ q_2}+{ m_1}\,{
q_1}}{{ m_1}+{ m_2}+{ m_3}}}, & v_4={\frac {{ \mu_3}\,{ p_4}}{{ \mu_4}}}-{\frac {{ m_4}\, \left( {
 p_3}+{ p_2}+{ p_1} \right) }{{ \mu_4}}}.
\end{array}$$
The new Hamiltonian is $$H_2(u_2,u_3,u_4,v_2,v_3,v_4)=\,{\frac {{{ v_2}}^{2}}{{2 M_2}}}+\,{\frac {{{ v_3}}^{2}}{{
2 M_3}}}+\,{\frac {{{ v_4}}^{2}}{{ 2 M_4}}}-U_2.$$
$U_2$ is the corresponding potential energy in the new coordinates and similarly $U_3, U_4$ in the below are the potential energy in the different cooordinates.The new Hamiltonian is independent of $g_4$ and $G_4$, the center of mass and total linear momentum respectively. This reduces the dimension by four from 16 to 12.\\
Next we change to symplectic polar coordinates to eliminate the integrals due to the angular momentum and rotational symmetry. Set
$$u_i=(r_i\cos(\theta_i), r_i\sin(\theta_i))$$ 
$$v_i=(R_i\cos(\theta_i)-\frac{\Theta_i}{r_i}\sin(\theta_i), R_i\sin(\theta_i)+\frac{\Theta_i}{r_i}\cos(\theta_i))$$
for $i=2,3,4$. Then the new Hamiltonian  becomes
$$H_3=\,{\frac {{{ R_2}}^{2}{{ r_2}}^{2}+{{ \Theta_2}}^{2}}{{2 M_2}
\,{{ r_2}}^{2}}}+\,{\frac {{{ R_3}}^{2}{{ r_3}}^{2}+{{
\Theta_3}}^{2}}{{ 2M_3}\,{{ r_3}}^{2}}}+\,{\frac {{{ R_4}}^{2}{{
 r_4}}^{2}+{{ \Theta_4}}^{2}}{{ 2M_4}\,{{ r_4}}^{2}}}-U_3.$$
Note that the Hamiltonian $H_3$ has only terms of difference angles. This suggests making a final symplectic change of coordinates by leaving the radial variables alone. Use the generating function $ S={ \Theta_2}\,{ x_2}+{ \Theta_3}\, \left( {x_3}+{ x_2} \right) +{ \Theta_4}\, \left( { x_4}+{ x_3}+{ x_2} \right)$, and so
 $$\theta_2 = x_2, \theta_3 = x_3+x_2, \theta_4 = x_4+x_3+x_2;$$
 $$ \Theta_2 = X_2-X_3, \Theta_3 = X_3-X_4; \Theta_4 = X_4.$$
 The new Hamiltonian will be independent of $x_2$ which means that $X_2=\Theta_2+\Theta_3+\Theta_4$ (total angular momentum) is an integral, and $x_2$ is an ignorable variable. Setting $X_2=c$ and plugging into the Hamiltonian $H_3$ yields
 $$H_4= \,{\frac {{{ R_2}}^{2}{{ r_2}}^{2}+ \left( {c}-{ X_3}
 \right) ^{2}}{{ 2M_2}\,{{ r_2}}^{2}}}+\,{\frac {{{ R_3}}^{2}{
{ r_3}}^{2}+ \left( { X_3}-{ X_4} \right) ^{2}}{{ 2M_3}\,{{
r_3}}^{2}}}+\,{\frac {{{ R_4}}^{2}{{ r_4}}^{2}+{{ X_4}}^{2}}{{
 2M_4}\,{{ r_4}}^{2}}}-U_4.$$
 This reduces the system to 10 dimensions, with the variables $z=(r_2, r_3, r_4, x_3, x_4, R_2, R_3, R_4,$ $ X_3, X_4)$. \\
 Because $H_4$ is a Hamiltonian system, the monodromy matrix $X(\mathcal{T})$ is symplectic. Its periodic solution $\gamma(t)$ will generate an eigenvector of $X(\mathcal{T})$. In fact, $\gamma(t)$ is a solution of $\dot{z}=J\nabla H_4(z)$ with initial condition $z(0)=\gamma(0)$. Then $\ddot{\gamma}(t)=JD^2H_4(\gamma(t))\dot{\gamma}(t)$. This implies that $\dot{\gamma}(t)$ satisfies the associated linear system
$$\dot{\xi}=JD^2H_4(\gamma(t))\xi,\hspace{0.2cm} \xi(0)=\dot{\gamma}(0).$$
Since $X(t)$ is the fundamental solution of the above linear system, $\dot{\gamma}(t)=X(t)\dot{\gamma}(0)$, which implies $X(\mathcal{T})\dot{\gamma}(0)=\dot{\gamma}(\mathcal{T})=\dot{\gamma}(0)$. Because $X(\mathcal{T})$ is symplectic, $J^{-1}X(\mathcal{T})J= X(\mathcal{T})$. Then $X(\mathcal{T})J\dot{\gamma}(0)=J\dot{\gamma}(0)$. So the Monodromy matrix has two $+1$ multipliers, leaving the remaining eight eigenvalues to determine the linear stability of the periodic solution. Because the eigenvalues of a symplectic matrix occur in quadruples $(\lambda,$ $ \lambda^{-1}, $ $\bar{\lambda}$, $\bar{\lambda}^{-1})$, we have the following lemma. 
\begin{lemma}
Let $X$ be a symplectic matrix and $W=\frac{1}{2}(X+X^{-1})$. Then the eigenvalues of $X$ are all on the unit circle if and only if all of the eigenvalues of $W$ are real and in $[-1,1]$.
\end{lemma}
\begin{proof}
The lemma and its proof are similar to Lemma 4.1 in Roberts' paper \cite{RG}. We prove it here for the sake of completeness. Suppose that $\vec{v}$ is an eigenvector of the symplectic matrix $X$ with eigenvalue $\lambda$, i.e. $X\vec{v}=\lambda \vec{v}$. Then $X^{-1}\vec{v}=\lambda^{-1}\vec{v}$.  $W\vec{v}=\frac{1}{2}(X+X^{-1})\vec{v}=\frac{1}{2}(\lambda+\lambda^{-1})\vec{v}$ from which it follows that $\frac{1}{2}(\lambda+\lambda^{-1})$ is an eigenvalue of $W$. The map $f : \mathcal{C}\mapsto\mathcal{C}$ given by $f(\lambda) = \frac{1}{2}(\lambda+\lambda^{-1})$ takes the unit circle onto the
real interval $[-1, 1]$ while mapping the exterior of the unit disk homeomorphically onto
$\mathcal{C}\backslash[-1, 1]$. The lemma follows this assertion immediately.
\end{proof}
Because the eigenvalue pairs $\lambda$ and $\lambda^{-1}$ of $X$ is mapped to the same eigenvalue $\frac{1}{2}(\lambda+\lambda^{-1})$ of $W$, the multiplicity of eigenvalues of $W$ must be at least two. The two $+1$ multipliers is still mapped to $+1$ with multiplicity two.  The remaining eight non-one eigenvalues on the unit circle of $X$ for linear stable periodic solution have been mapped to four pairs of real eigenvalues in $(-1,1)$.

Numerically, a MATLAB program was written using a Runge-Kutta-Fehlberg method to compute the monodromy matrix $X(\mathcal{T})$ of the  reduced linearized Hamiltonian $H_4$ for the star pentagon choreographic solution presented in figure \ref{fig1}. Then we compute $W=\frac{1}{2}(X+X^{-1})$ and  its eigenvalues.




$\theta=\frac{2\pi}{5}$, $\mathcal{T}=40$, the four pairs of eigenvalues other than the multipliers $+1$ are $[ 0.761537,$  $  0.761537, $  $ 0.235841,$  $ 0.235841,$  $  - 0.299445, $  $ - 0.299445, $  $ - 0.456736, $  $- 0.456736].$ 
 The four pairs of eigenvalues are real and distinct  in $(-1,1)$. Returning to the full monodromy matrix, the corresponding eigenvalues are distinct and on the unit circle. Therefore, the star pentagon choreographic solution is linearly stable.

\section{Classification of solutions and rotation angles $\theta$} \label{sec5}

Recall that a {\it simple choreographic solution} (for short, choreographic solution) is a periodic solution that all bodies chase one another along a single orbit.  If the orbit of a periodic solution consists of two
 closed curves, then it is
called a {\it double-choreographic solution}. If the orbit of a periodic solution consists of different closed curves, each of which is the trajectory of exact one body,  it is called {\it non-choreographic solution}. Theorem \ref{Main2}  can be restated as following theorem with detail classifications.
\begin{theorem}\label{Thm:51}

There exist $\theta_0$ and $\theta_1$  such that $\frac{\pi}{4}<\theta_0<\frac{\pi}{2}<\theta_1<\frac{3\pi}{4}$. For any $\theta\in(\theta_0,\theta_1)$ and $\theta\not=\frac{\pi}{2}$, there exists at least one local minimizer $\vec{a}_0\in \Gamma$ for  the variational problem \eqref{VAR} with the SPBC and its corresponding minimizing path $q^*(t)$ connecting $q(0)$ and $q(T)$ can be extended to a non-circular classical Newtonian  solution $q(t)$ by the same extension as \eqref{qet} in theorem \ref{Thm:ET}.  Each curve $q_i(t), t\in[8kT,(8k+8)T]$ is called a side of the orbit since the orbit of the solution is assembled out the sides by rotation only. The non-circular solution $q(t)$ can be classified as follows.\\
 (1) {\bf [Quasi-Periodic Solutions]}  $q(t)$ is a quasi-periodic solution if $\theta$ is not commensurable with $\pi$. \\
 (2) {\bf [Periodic Solutions]} $q(t)$ is a periodic solution if $\theta=\frac{P}{Q}\pi$, where the positive integers $P$ and $Q$ are relatively prime.
\begin{itemize}
\item  When $Q \equiv 0 \mod 4 $, the periodic solution $q(t)$ is a non-choreographic solution. Each closed curve has $\frac{Q}{4}$ sides. The minimum period is $\mathcal{T}=2QT$.
\item When $Q\equiv 1 \mod 4$,  the periodic solution $q(t)$ is a choreographic solution. The closed curve has $Q$ sides. The minimum period is $\mathcal{T}=8QT$. The four bodies chase each other on the closed curve in the order of $q_1, q_2, q_3, q_4,$ and then $q_1$, i.e. $q_1(t+2QT)=q_2(t),$ $q_2(t+2QT)=q_3(t),$ $q_3(t+2QT)=q_4(t),$ and $q_4(t+2QT)=q_1(t).$
 \item When $Q\equiv 2 \mod 4$,  the periodic solution $q(t)$ is a double-choreographic solution. Each closed curve has $\frac{Q}{2}$ sides. The minimum period is $\mathcal{T}=4QT$. Body $q_1$ chase body $q_3$ on a closed curve and body $q_2$ chase body $q_4$ on another closed curve. $q_1(t+2QT)=q_3(t)$ and $q_3(t+2QT)=q_1(t).$ $q_4(t+2QT)=q_2(t)$ and $q_2(t+2QT)=q_4(t).$
 \item When $Q\equiv 3 \mod 4$,  the periodic solution $q(t)$ is a choreographic solution. The closed curve has $Q$ sides. The minimum period is $\mathcal{T}=8QT$. The four bodies chase each other on the closed curve in the order of $q_1, q_4, q_3, q_2,$ and then $q_1$, i.e. $q_1(t+2QT)=q_4(t),$ $q_4(t+2QT)=q_3(t),$ $q_3(t+2QT)=q_2(t),$ and $q_2(t+2QT)=q_1(t).$

\end{itemize}
(3) {\bf [Linear Stability]} If $\theta=\frac{P}{2P+1}\pi$, the simple choreographic solutions $q(t)$ are linearly stable for $P=2,3,4,\cdots, 15$. If $\theta=\frac{2P-1}{4P},$ the non-choreographic solutions $q(t)$ are linearly stable for $P=3,4,\cdots,8$. If $\theta=\frac{2P-1}{4P+2},$ the double choreographic solutions $q(t)$ are linearly stable for $P=5,6, 7$.
\end{theorem}

\begin{proof} We first observe  that the proposition \ref{PROP:1} of existence of minimizers in space $\Gamma$ still holds for $\theta\in(\theta_0,\theta_1)$ and $\theta\not=\frac{\pi}{2}$ by remark \ref{remark2}. The proof of the non-collision theorem \ref{Thm:NC} does not depend on $\theta$ and the extension property \eqref{qet} in theorem \ref{Thm:ET} is also independent of the rotation angle $\theta$.
Now for given $\theta$, we compare the action of the test path $\bar{q}(t)$ with constant velocity for $\vec{a} =[   1.0597,$ $    1.7696,$  $  0.8094,$ $  0.7536,$ $    1.1032,$ $    2.4398]$  and the action of the path $q^\circ(t)$ for $\vec{a}^\circ$ which is extended to a circular motion. The test path is constructed by connecting $q(0)$ and $q(T)$ by straight line segment with constant velocity. Both actions $\mathcal{A}(q^\circ(t))$ and $\mathcal{A}(\bar{q}(t))$ are explicit continuous functions of $\theta$  given by formula \eqref{actCir} and \eqref{actTest} respectively (see dashed line and dashdotted line in Figure \ref{fig2}). Both calculations of the functions  are provided in appendix A and appendix B.
\begin{figure} 
\includegraphics[height=7cm,width=.8\textwidth]{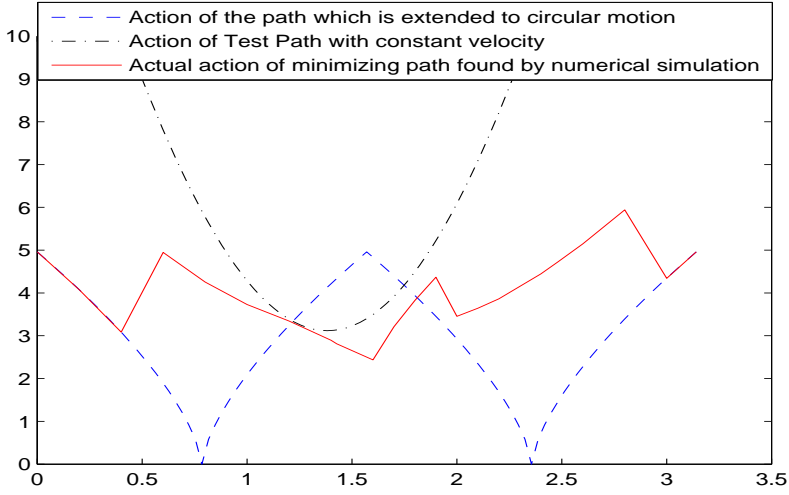}
\caption{Action levels of test path with constant velocity (black), circular motion (blue), and minimizing path (red).}
\label{fig2}\end{figure}
So there exist $\theta_0$ and $\theta_1$ such that for $\theta\in (\theta_0,\theta_1)$, the action of the test path is smaller than the action of the path $q^\circ(t)$. Numerically, $1.1938<\theta_0<1.2252$ or $0.38\pi<\theta_0<0.39\pi$. $1.7279<\theta_1<1.7593$ or $0.55\pi<\theta_1<0.56\pi$. Therefore there exists a local minimizer with smaller action and the corresponding minimizing path $q^*(t)$ on $[0,T]$ can be extended to a non-circular solution $q(t)$ as in theorem \ref{Thm:ET} and $$q(t)=\sigma^{k}(q(t-2kT))R(2k\theta) \hbox{ for } t\in (2kT,(2k+2)T]
  \hbox{ and } k\in \mathbf{Z}^+.$$
   (1) By the extension formula, it is easy to show that $q(t)$ is a quasi-periodic solution if $\theta$ is not commensurable with $\pi$. \\
 (2) If $\theta$ is commensurable with $\pi$ and $\theta=\frac{P}{Q}\pi$ where the positive integers $P$ and $Q$ are relatively prime, then $q(t)=\sigma^{4Q}(q(t-8QT))R(2P\pi)=q(t-8QT)$ which implies that $q(t)$ is a periodic solution.
 For $0\leq k<Q$, the trajectory sets $\{q_i(t)| t\in(2k,(2k+2)T)\}$  $i=1,2,3,$ and $4$ are all different since the rotation matrix $R(2k\frac{P}{Q}\pi)$ is not identity matrix. The four trajectories on which the four body travel  in $t\in(0,2QT)$ are all different.
 \begin{itemize}
 \item When $Q\equiv 0 \mod 4$, $\sigma^{Q}=\sigma^0=[1,2,3,4]$ and $q(t)=q(t-2QT)$. So the four different trajectories in $t\in(0,2QT)$  are closed on their own at $t=2QT$, i.e. $q(2QT)=q(0)$. The periodic solution  is non-choreographic and the minimum period is $\mathcal{T}=2QT$. Each close curve has $\frac{2QT}{8T}=\frac{Q}{4}$ sides. In particular, when $Q=4$, each closed curve is circle-like (one side); when $Q=8$, each closed curve is ellipse-like (two sides); when $Q=12$, each closed curve is triangle-like (three sides); and so on. See figure \ref{fig5}.
 \item When $Q\equiv 1 \mod 4$, $\sigma^{Q}=\sigma^{1}=[2,3,4,1]$ and $q(t)=\sigma(q(t-2QT))$. The four different trajectories in $t\in(0,2QT)$ are connected at $t=2QT$ as $q_1(2QT)=q_2(0)$, $q_2(2QT)=q_3(0)$, $q_3(2QT)=q_4(0)$, and $q_4(2QT)=q_1(0)$, and they form a closed orbit. Since $\sigma^{4Q}=[1,2,3,4]$ and $q(t)=q(t-8QT)$, $q(t)$ is a simple choreographic solution with minimum period $\mathcal{T}=8QT$ and it has $\frac{8QT}{8T}=Q$ sides. In particular, when $Q=5$ and $P=2$, the orbit is a star pentagon (five sides) (See figure \ref{fig1}). When $Q=9$ and $P=4$, the orbit is a star nonagon (nine sides). When $Q=13$ and $P=6$, the orbit is a star tridecagon (thirteen sides); and so on. See figure \ref{fig6}.

 \item  When $Q\equiv 2 \mod 4$, $\sigma^{Q}=\sigma^{2}=[3,4,1,2]$ and $q(t)=\sigma^2(q(t-2QT))$. The four different trajectories in $t\in(0,2QT)$ are connected at $t=2QT$ in two pairs, i.e. $q_1(2QT)=q_3(0)$ and $q_3(2QT)=q_1(0)$,  $q_2(2QT)=q_4(0)$ and $q_4(2QT)=q_2(0)$. So they form two closed orbits. Since $\sigma^{2Q}=[1,2,3,4]$ and $q(t)=q(t-4QT)$, $q(t)$ is a double-choreographic solution with minimum period $\mathcal{T}=4QT$ and each closed orbit has $\frac{4QT}{8T}=\frac{Q}{2}$ sides. In particular, when $Q=10$ and $P=3$, the orbit is the combination of two flowers with five petals each. See figure \ref{fig7}

     \item When $Q\equiv 3 \mod 4$, $\sigma^{Q}=\sigma^{3}=[4,1,2,3]$ and $q(t)=\sigma^3(q(t-2QT))$. The four different trajectories in $t\in(0,2QT)$ are connected at $t=2QT$ as $q_1(2QT)=q_4(0)$, $q_4(2QT)=q_3(0)$, $q_3(2QT)=q_2(0)$, and $q_2(2QT)=q_1(0)$, and they form a closed orbit. Since $\sigma^{4Q}=[1,2,3,4]$ and $q(t)=q(t-8QT)$, $q(t)$ is a simple choreographic solution with minimum period $\mathcal{T}=8QT$ and it has $\frac{8QT}{8T}=Q$ sides. In particular, when $Q=7$ and $P=3$, the orbit is a star heptagon (seven sides). See figure \ref{fig8}.

\end{itemize}
(3) We numerically compute the Monodromy matrix $X(\mathcal{T})$ of the reduced Hamiltonian system $H_4$ as in section \ref{sec4}. Then we compute $W=\frac{1}{2}(X+X^{-1})$ and  its eigenvalues for the case of  $\theta=\frac{P}{2P+1}\pi$, $P=2,3,4,\cdots, 15$ and we list for $P=3,4,15$ here.  The four pairs of eigenvalues are real and distinct  in $(-1,1)$. Returning to the full monodromy matrix, the corresponding eigenvalues are distinct and on the unit circle. Therefore, the corresponding choreographic solutions are linearly stable. By the same way, we prove the stability for other cases and we list the results for $\theta=\frac{5\pi}{12}, $ and $ \frac{9\pi}{22}.$  


$\theta=\frac{3\pi}{7}$, $\mathcal{T}=56$, the four pairs of eigenvalues other than the multipliers $+1$ are $ [  - 0.375476, $ $ - 0.375476, $ $ 0.493924, $ $   0.493924, $ $  0.623185, $ $  0.623185, $ $  0.698755,$ $  0.698755].$ 


$\theta=\frac{4\pi}{9}$, $\mathcal{T}=72$, the four pairs of eigenvalues other than the multipliers $+1$ are $[- 0.888315, $ $ - 0.888315, $ $ 0.717492,$ $   0.717492,$ $  0.781167, $ $  0.781167,$ $  0.875241, $ $  0.875241].$ 


 $\theta=\frac{15\pi}{31}$, $\mathcal{T}=248$, the four pairs of eigenvalues other than the multipliers $+1$ are  $[  - 0.761943,$ $   - 0.761943, $ $  - 0.0535079, $ $  - 0.0535079, $ $   - 0.375899,$ $    - 0.375899, $ $ 0.994815,$ $   0.994815].$ 


$\theta=\frac{5\pi}{12}$, $\mathcal{T}=24$, the four pairs of eigenvalues other than the multipliers $+1$ are  $[ - 0.752385,$ $  - 0.752385, $ $ 0.786314, $ $  0.786314,$ $  0.850377, $ $  0.850377, $ $ 0.845072, $ $  0.845072].$ 


 $\theta=\frac{9\pi}{22}$, $\mathcal{T}=88$, the four pairs of eigenvalues other than the multipliers $+1$ are  $[- 0.612649,$ $ - 0.612649, $ $ - 0.791503, $ $- 0.791503,$ $  - 0.967491, $ $ - 0.967491, $ $ - 0.99911,$ $  - 0.99911]. $ 


\end{proof}


\section*{Appendix A: Action of the path which is extended to a circular solution}
The configuration $q$ is called a central configuration if $q$ satisfies the following nonlinear
algebraic equation system:
\begin{equation}\label{CC}
\lambda (q_i-c)-\sum_{j=1,j\not= i}^{n} \frac{m_j(q_i-q_j)
}{|q_i-q_j|^3}=0, \hspace{1cm} 1\leq i\leq n,
\end{equation}
for a constant $\lambda$, where $c=(\sum m_i q_i)/M$ is the center
of mass and $M=m_1+m_2+\cdots +m_n$ is the total mass. We recall the fact that coplanar central configurations always admit homographic solutions where each body executes a similar Keplerian ellipse of eccentricity $e$, $0\leq e\leq 1$. When $e=0$, the relative equilibrium solutions are consisting  of uniform circular motion for each of the masses about the common center of mass. When $e=1$, the homographic solutions degenerate to a homothetic solution which includes total collision, together with a symmetric segment of ejection. Gordon found that for fixed period $\mathcal{T}^{\circ}$, all of the homographic solutions have the same action (\cite{GW}).  Consider the circular solution of the four-body problem with equal masses  $$q_{k}^\circ (t)=r(cos(\omega t+\rho_k),sin(\omega t+\rho_k)), k=1, 2,3, 4,$$
where $r>0$ is the radius of the circle and $\rho_k=\frac{k\pi}{2}$. We can easily find the relation between $\omega$ and $r$ by Newtonian equations \eqref{Newton}:
$$r^3\omega^2=\frac{1}{4}U_0,$$
where $U_0=2\sqrt{2}+1$ is the potential energy of the four equal masses at the square configuration on a unit circle. The minimum period is $\mathcal{T}^{\circ}=\frac{2\pi}{\omega}$.  The minimum value of the  action functional \eqref{Min} (realized by the circular solution) could be computed $$\mathcal{A}(q(t))=\int_0^{\mathcal{T}^\circ}\sum_{k=1}^4\frac{1}{2}m_k|\dot{q}_k(t)|^2 +U(q(t))dt =(2r^2\omega^2+\frac{U_0}{r}){\mathcal{T}^\circ}$$ $$=\frac{3 U_0}{2 r}{\mathcal{T}^\circ}=3(2)^{\frac{1}{3}}U_0^{\frac{2}{3}}\pi^{\frac{2}{3}}{(\mathcal{T}^\circ)}^{\frac{1}{3}}.$$

  \begin{figure}
\includegraphics[height=6cm,width=.45\textwidth]{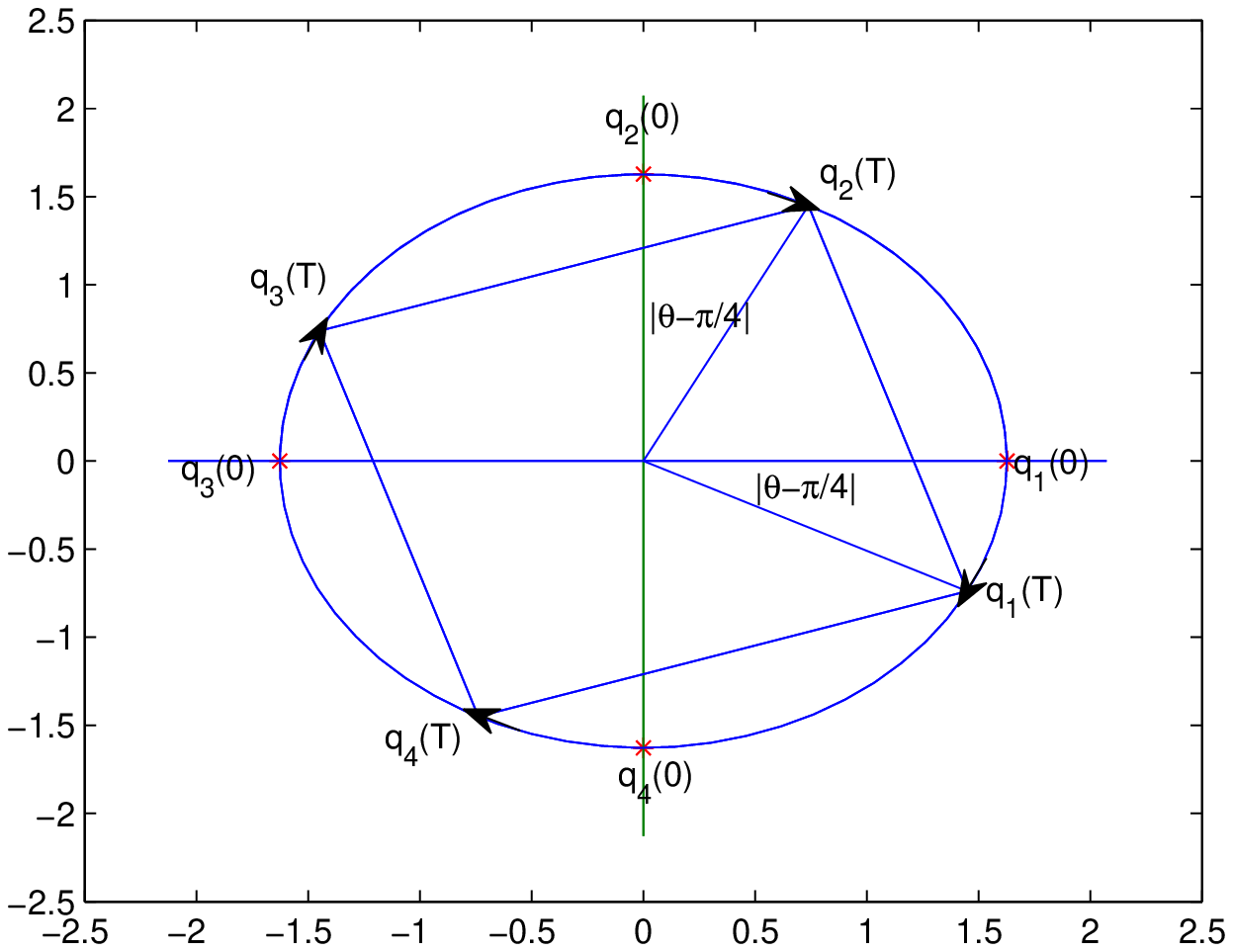}
\includegraphics[height=6cm,width=.45\textwidth]{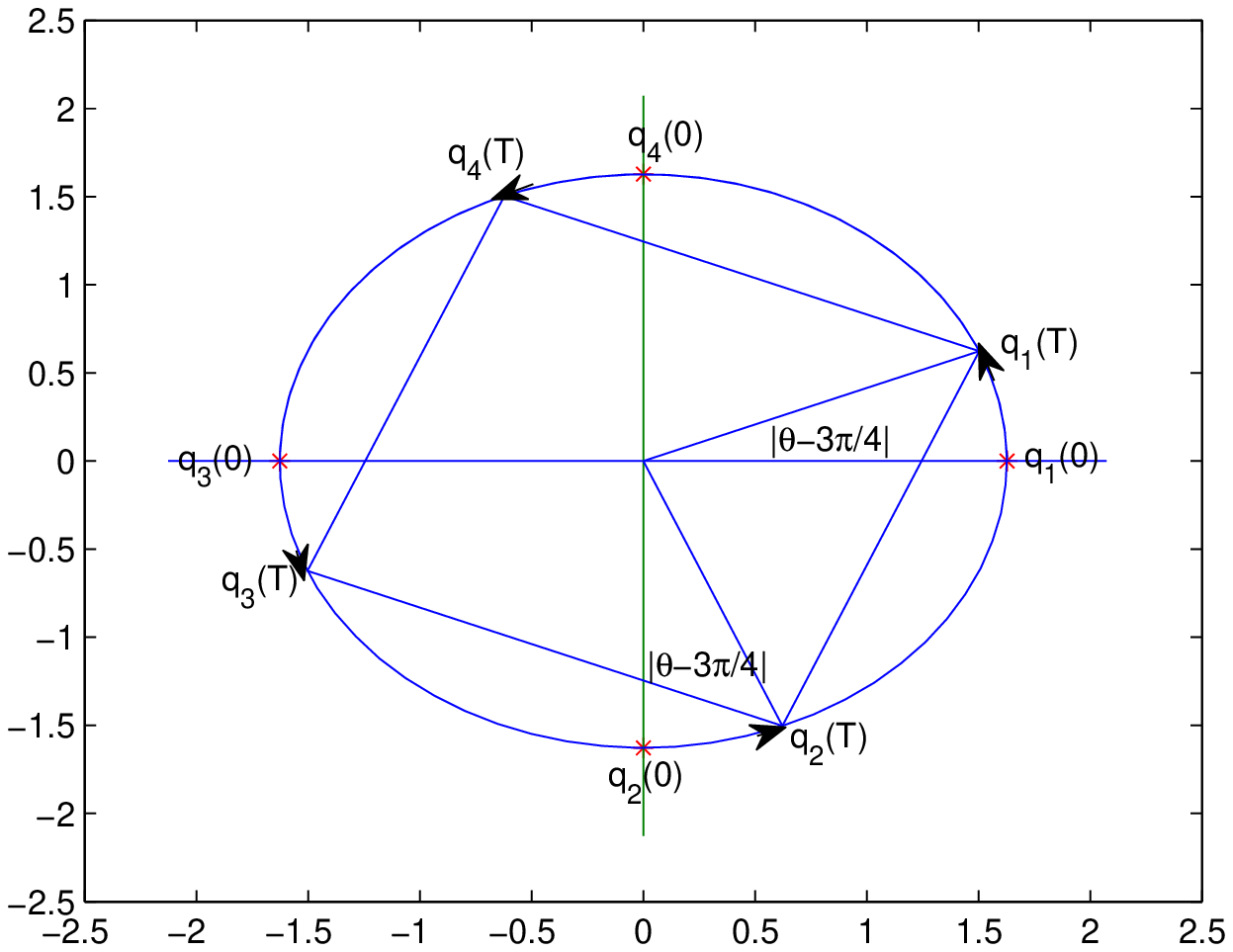}
\caption{ Left: For $\frac{\pi}{4}<\theta<\frac{\pi}{2}$, $\vec{a}^{\circ}=[a_1, 0, -a_1, \frac{\sqrt{2}a_1}{2},-\frac{\sqrt{2}a_1}{2},\frac{\sqrt{2}a_1}{2}]$ with $a_1>0$ and $\alpha=|\theta-\frac{\pi}{4}|$. Right: For $\frac{\pi}{2}<\theta<\frac{3\pi}{4}$, $\vec{a}^{\circ}=[a_1, 0, a_1, \frac{\sqrt{2}a_1}{2},\frac{\sqrt{2}a_1}{2},-\frac{\sqrt{2}a_1}{2}]$ with $a_1>0$ and $\alpha=|\theta-\frac{3\pi}{4}|$.
  }\label{Apa}\end{figure}
If the bodies rotate an angle $\alpha$ in the time interval $[0,T]$, $\omega=\frac{\alpha}{T}$ and $\mathcal{T}^\circ=\frac{2\pi T}{\alpha}$ and the action over the time interval $[0,T]$ is
$$ 3(2)^{\frac{1}{3}}U_0^{\frac{2}{3}}\pi^{\frac{2}{3}}{(\mathcal{T}^\circ)}^{\frac{1}{3}}\frac{\alpha}{2\pi}= 3(2)^{-\frac{1}{3}}U_0^{\frac{2}{3}}{T}^{\frac{1}{3}}\alpha^{\frac{2}{3}}\approx 5.8271766 {T}^{\frac{1}{3}}\alpha^{\frac{2}{3}}.$$
For fixed $T>0$, the configurations generating a circular solution should be always a square  and the bodies rotate an angle $\alpha$ as small as possible in order to have a minimizing action over the time interval $[0,T]$. By the structure of our prescribed boundary conditions, if $\frac{\pi}{4}<\theta<\frac{\pi}{2}$, the angle that the bodies rotate in the time interval $[0,T]$ is $\alpha=\theta-\frac{\pi}{4}$ (see the left graph of Figure \ref{Apa}). And  the SPBC is given by  $\vec{a}^{\circ}=[a_1, 0, -a_1, \frac{\sqrt{2}a_1}{2},-\frac{\sqrt{2}a_1}{2},\frac{\sqrt{2}a_1}{2}]$  which generates a circular solution with $r=a_1, \omega=\frac{\theta-\frac{\pi}{4}}{T}$.  The radius $a_1=U_0^{1/3}T^{2/3}(2\alpha)^{-2/3}$ of the circular solution is uniquely determined by $T$ and the corresponding angle.
 So the exact formula for action $\mathcal{A}_{\vec{a}^{\circ}}(q^{\circ}(t))$ of the path $q^{\circ}(t)$ in the time $[0, T]$ which can generate a circular motion with period $\mathcal{T}^\circ$ is
\begin{equation}\label{actCir}
\mathcal{A}(q^{\circ}(t))= 3(2)^{-\frac{1}{3}}U_0^{\frac{2}{3}}{T}^{\frac{1}{3}}\left( \theta-\frac{\pi}{4}  \right)^{\frac{2}{3}},
\end{equation}
where $\mathcal{T}^{\circ}=\frac{2\pi T}{\theta-\frac{\pi}{4}}$. For $T=1$ and $\theta=\frac{2\pi}{5}$, $a_1=1.6272$, and  $$\mathcal{A}(q^{\circ}(t))\approx 3.528734094.$$
If $\frac{\pi}{2}<\theta<\frac{3\pi}{4}$, the angle that the bodies rotate in the time interval $[0,T]$ is $\alpha=|\theta-\frac{3\pi}{4}|$ (see the right graph of Figure \ref{Apa}). And  the SPBC is given by  $\vec{a}^{\circ}=[a_1, 0, a_1, \frac{\sqrt{2}a_1}{2},\frac{\sqrt{2}a_1}{2},-\frac{\sqrt{2}a_1}{2}]$. For other $\theta$, we have similar results. If $\frac{3\pi}{4}<\theta<\pi$, $\alpha=|\theta-\frac{3\pi}{4}|$. If $0<\theta<\frac{\pi}{4}$, $\alpha=|\theta-\frac{\pi}{4}|$.

\section*{Appendix B: Action of Test Path  with Constant Velocity}

Given $\theta, T $ and $\vec{a}=(a_1,a_2,\cdots,a_6)$, the test path $\bar{q}(t)$ with constant velocity connecting the structural prescribed boundary conditions is given by

$$\bar{q}(t)=Qstart+\frac{t(Qend-Qstart)}{T}, t\in[0,T].$$
\begin{equation}\label{actTest}
 \begin{array}{ll}
& \mathcal{A}(\bar{q}(t))= \sum_{k=1}^4\frac{1}{2T}m_k\|Qend_k-Qstart_k\|^2 \\
\\
  &+\int_0^T  \sum_{1\leq k<j\leq 4} \frac{m_k m_j}{\|(Qstart_k-Qstart_j)(1-\frac{t}{T})+(Qend_k-Qend_j)\frac{t}{T}\|}dt,
 \end{array} \end{equation}
where the integrals only involve the form of $\int_0^T\frac{m_km_j}{(a+bt+ct^2)^{1/2}} dt $ which can be integrated explicitly by trigonometric substitution. In figure \ref{fig2}, the action of test path is computed by assuming $\vec{a}=[1.0597,$ $    1.7696,$ $    0.8094,$ $    0.7536,$ $    1.1032,$ $    2.4398]$ with $T=1$ and the action is an explicit function of $\theta$. 
For example, for $T=1$ and $\theta=\frac{2\pi}{5}$, $\mathcal{A}(\bar{q}(t))= 1.0633+ \int_0^1 \frac{1}{(7.7742-6.1698*t+3.2638*t^2)^{1/2}}+\cdots +\frac{1}{(21.3675+2.22208*t+.2208*t^2)^{1/2}}=3.2484.$

\section*{Appendix C: Numerical simulations for the orbits with different rotation angle $\theta$.}
Here we present some numerical simulations for the orbits with different rotation angle $\theta$. All the orbits are extended from the initial four pieces connecting from $q(0)$ to $q(T)$ by extension formula \eqref{qet}. We use $T=1$ in our calculation. As we point out in Remark \ref{rem2}, non-circular minimizers exist for $\theta$ out of the interval $[\theta_0,\theta_1]$. Most of the figures can be generated in any Newtonian $n$-body simulation program by using these initial data. However, some examples are highly unstable and it is hard to produce satisfactory numerical figures. We list the initial conditions for some stable orbits.

$\theta=\frac{3\pi}{7}$,  $q_1(0)= [  0.9421459089,  2.189431278], $ $\dot{q}_1(0)=[ -0.4908870906,      -0.474846006]$,
 $q_2(0)=[             0, -1.300514651], $ $\dot{q}_2(0)=[1.039544889, 0],$  $q_3(0)= [ -0.9421459089,  2.189431278],$ $\dot{q}_3(0)=[ -0.4908437595,      0.4748326024],$ $q_4(0)= [             0, -3.078347905],$ $\dot{q}_4(0)=[ -0.05781403894,   0].$

$\theta=\frac{5\pi}{12}$,  $q_1(0)=[  0.9885667998,  1.984831768], $ $\dot{q}_1(0)=[-0.5187341985,      -0.4460463326]$,
 $q_2(0)=[ 0, -1.067317853], $ $\dot{q}_2(0)=[  1.064058961, 0],$  $q_3(0)= [  -0.9885667998,  1.984831768],$ $\dot{q}_3(0)=[ -0.5187107584,       0.4460354291],$ $q_4(0)= [             0, -2.902345684],$ $\dot{q}_4(0)=[  -0.02661400412,    0].$

 $\theta=\frac{9\pi}{22}$,  $q_1(0)=[  1.020078100,   1.878808307], $ $\dot{q}_1(0)=[  -0.5352327448,     -0.4256795762]$, $q_2(0)=[    0, -0.9422097516], $ $\dot{q}_2(0)=[  1.078223783,0],$  $q_3(0)= [ -1.02007810,   1.878808307],$ $\dot{q}_3(0)=[-0.5352124256,      0.4256692972],$ $q_4(0)= [          0,  -2.815406862],$ $\dot{q}_4(0)=[   -0.007778613051,   0].$

\begin{figure}
\includegraphics[height=5cm,width=.32\textwidth]{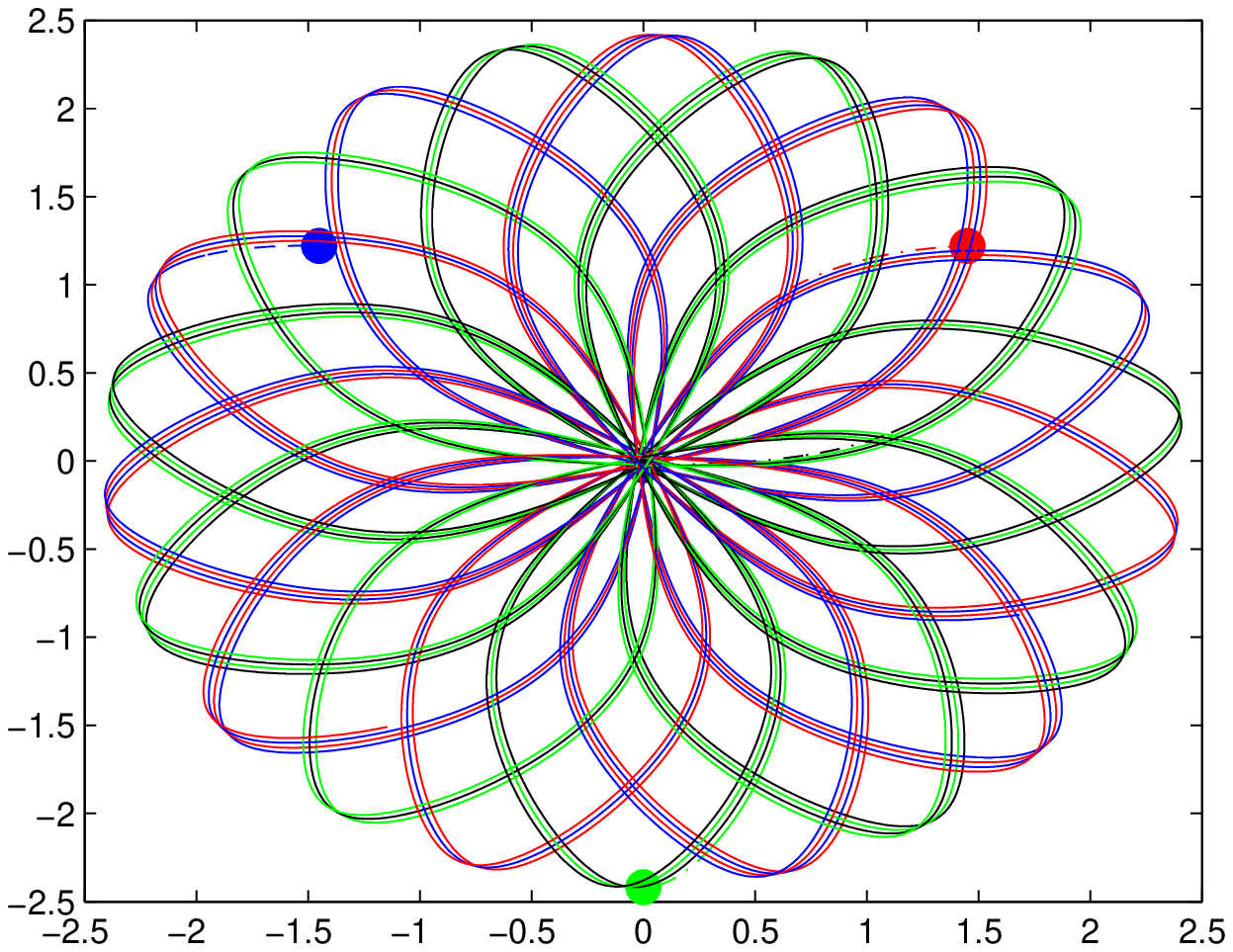}
\includegraphics[height=5cm,width=.32\textwidth]{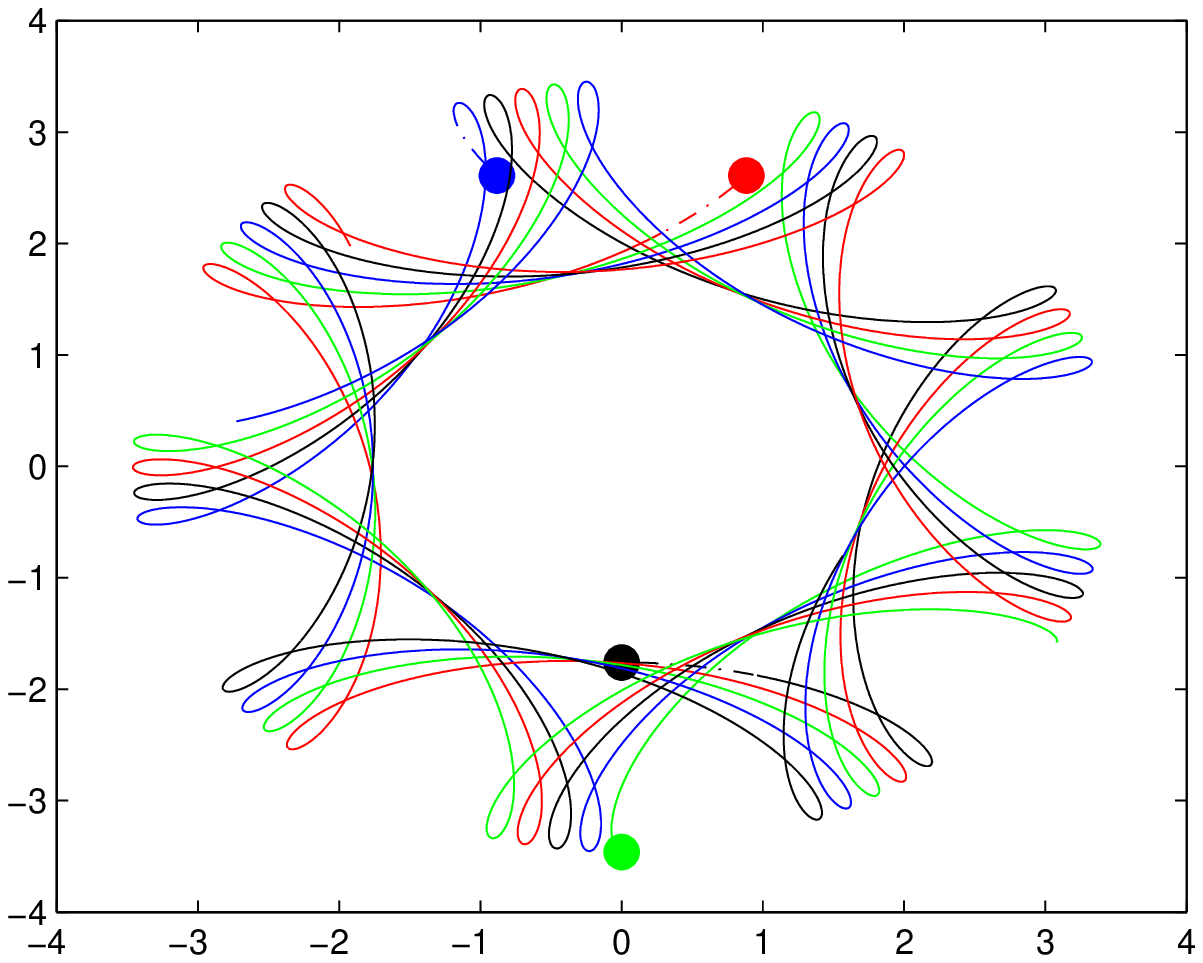}
\includegraphics[height=5cm,width=.32\textwidth]{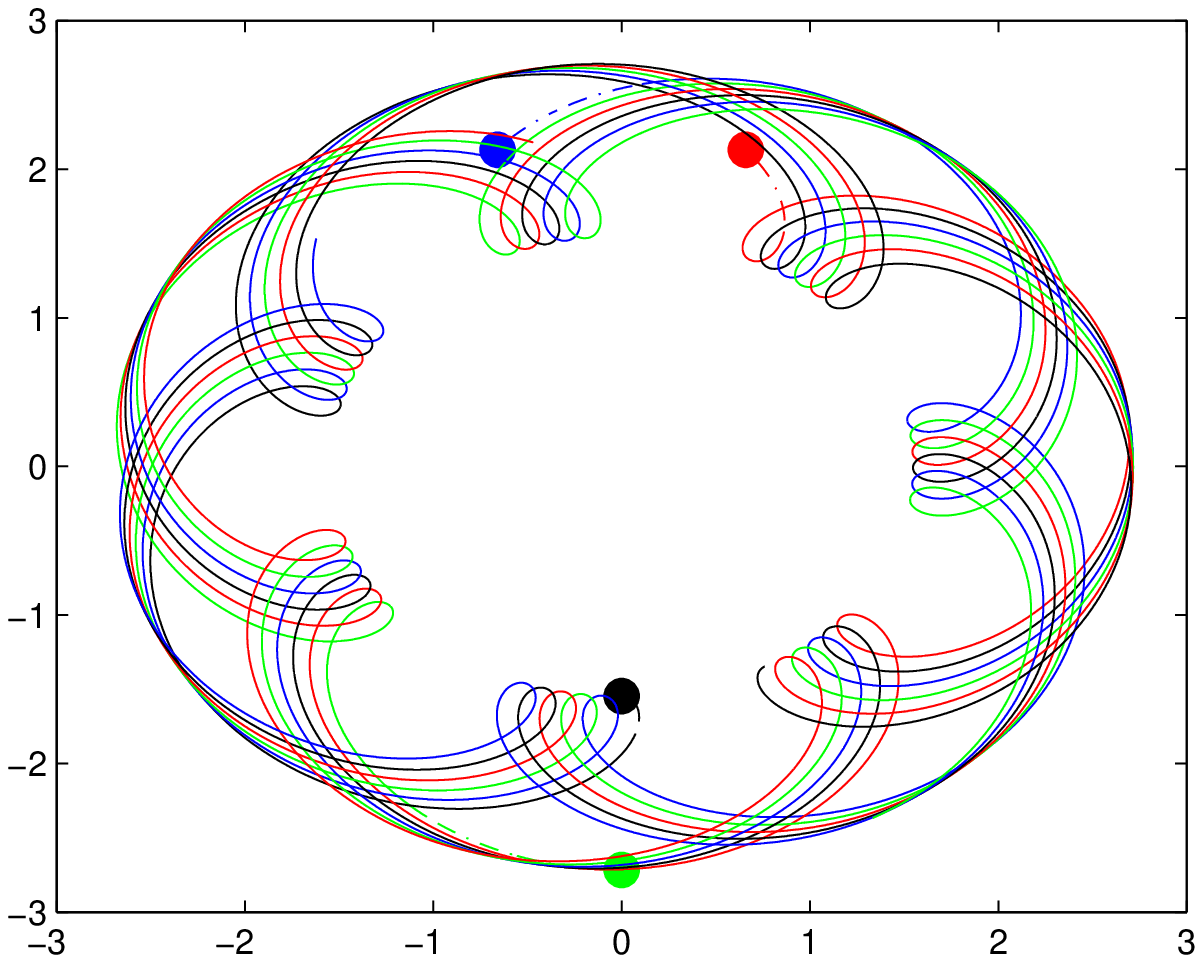}
\caption{\small Quasi-Periodic Solutions on $[0,80T]$. From left to right $\theta=1.0, 1.4, 1.8.$ }
\label{fig9}\end{figure}

\begin{figure}
\includegraphics[height=5cm,width=.32\textwidth]{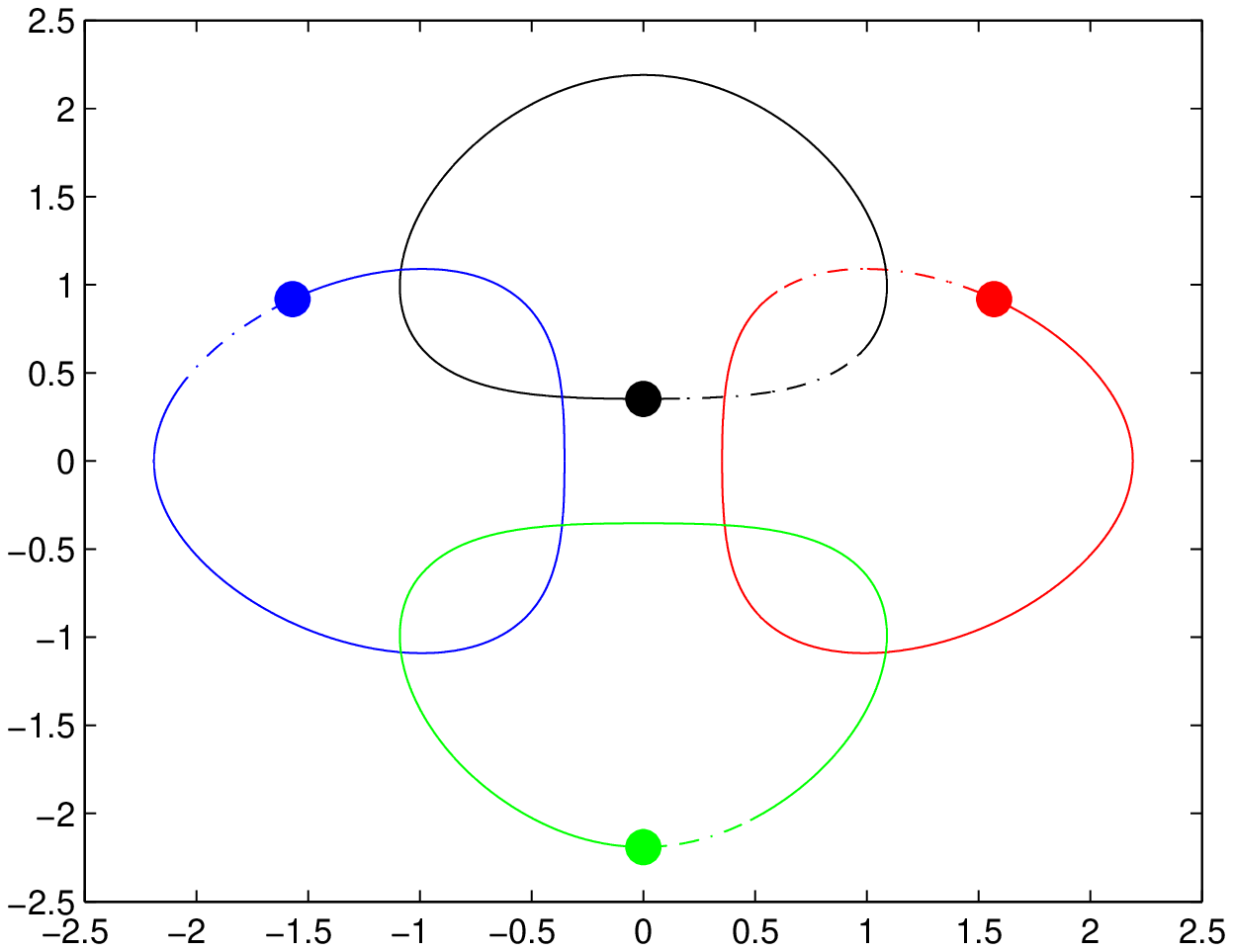}
\includegraphics[height=5cm,width=.32\textwidth]{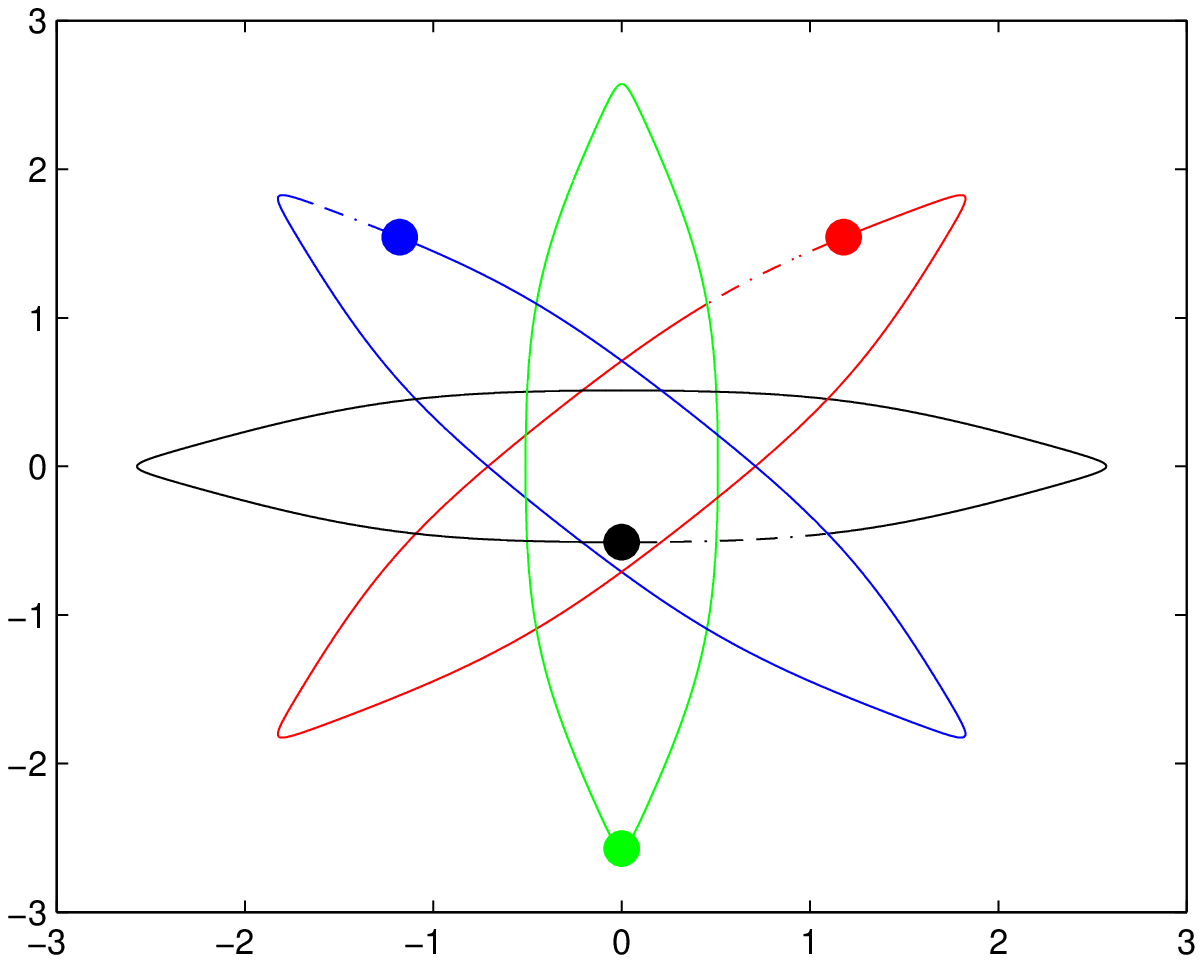}
\includegraphics[height=5cm,width=.32\textwidth]{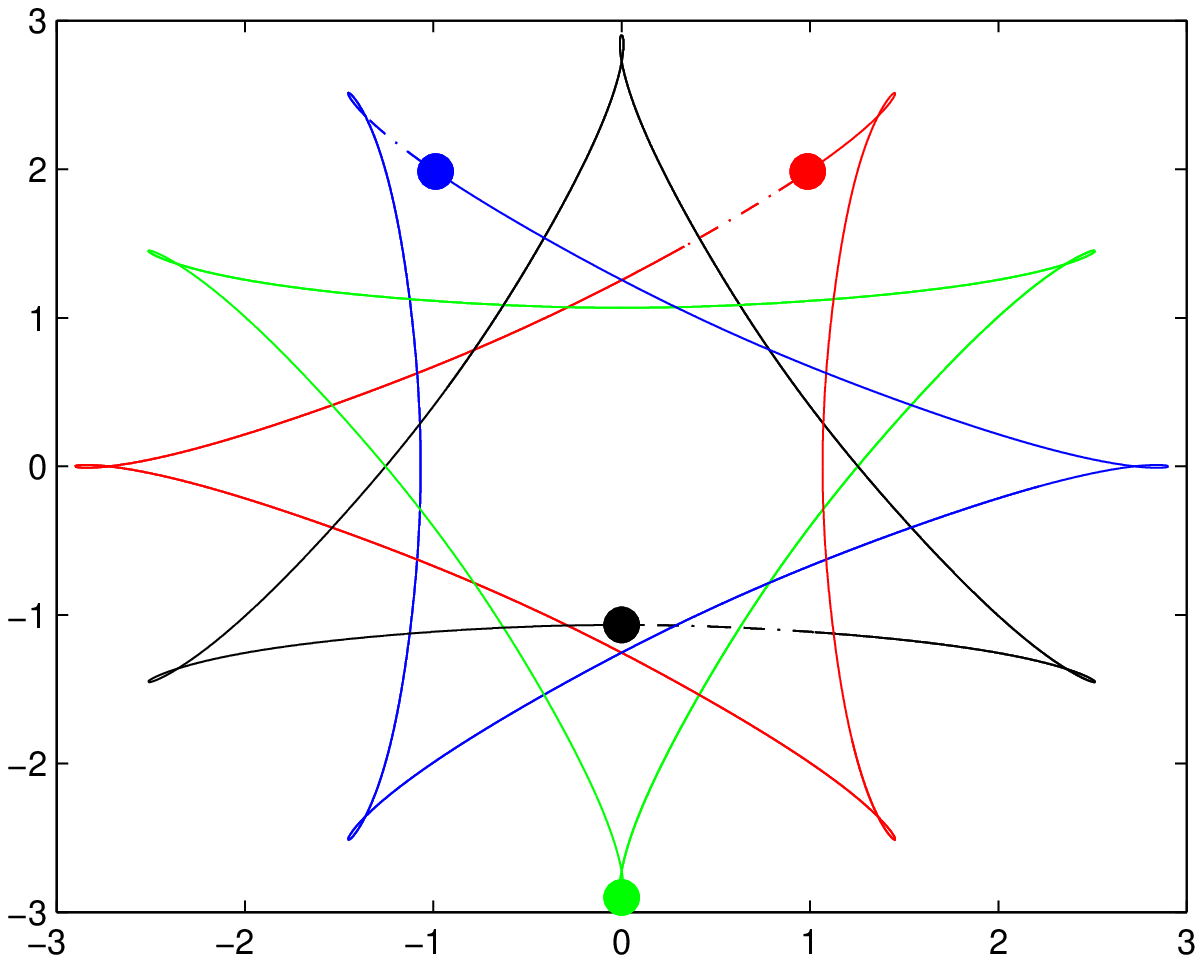}
\caption{\small Non-Choreographic Periodic Solutions  for $\theta = \frac{P}{Q}\pi$ with $Q\equiv 0 \mod 4$. From left to right $\theta= \frac{\pi}{4},$ $\frac{3\pi}{8},$ and  $\frac{5\pi}{12}$. }
\label{fig5}\end{figure}

\begin{figure}
\includegraphics[height=5cm,width=.32\textwidth]{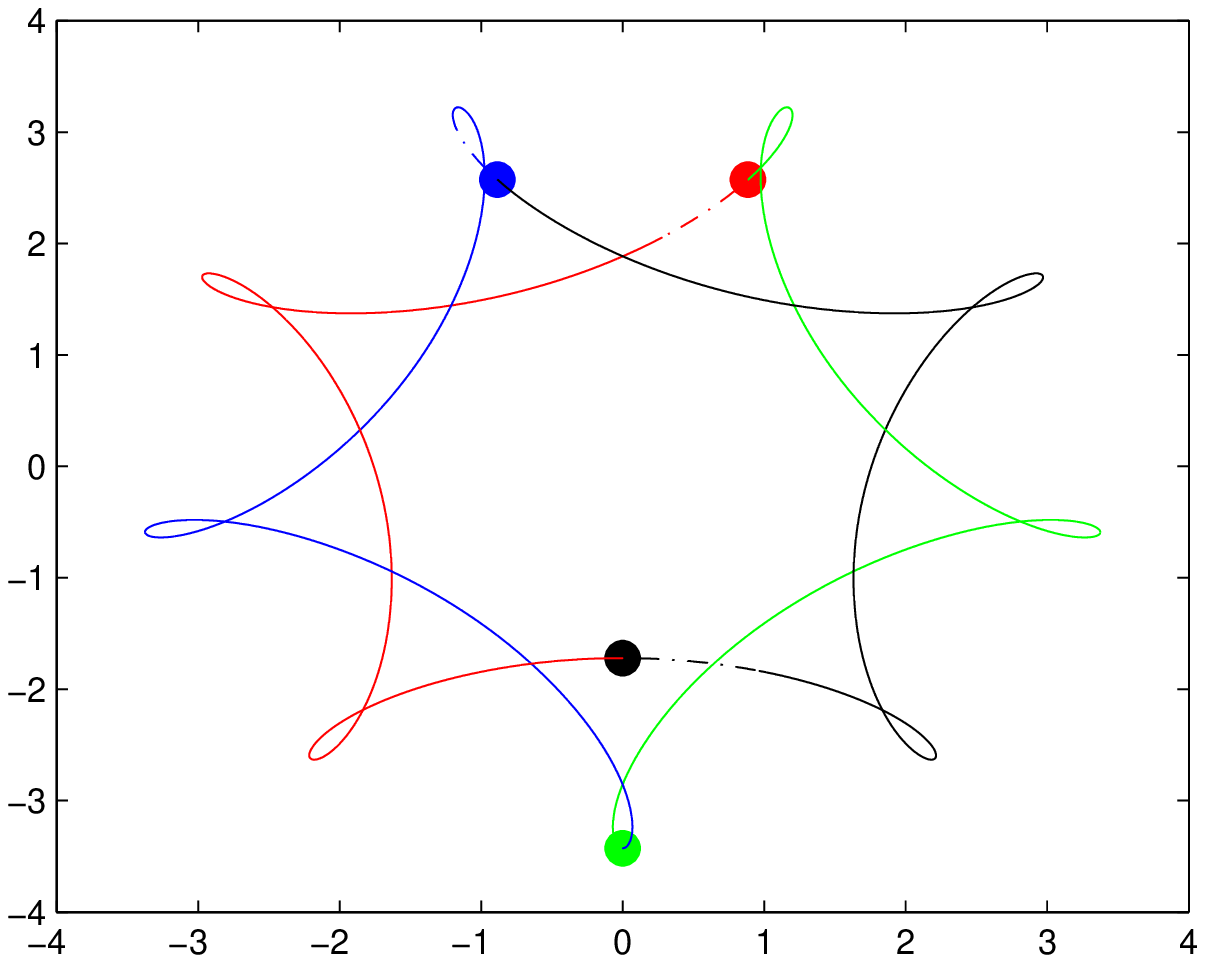}
\includegraphics[height=5cm,width=.32\textwidth]{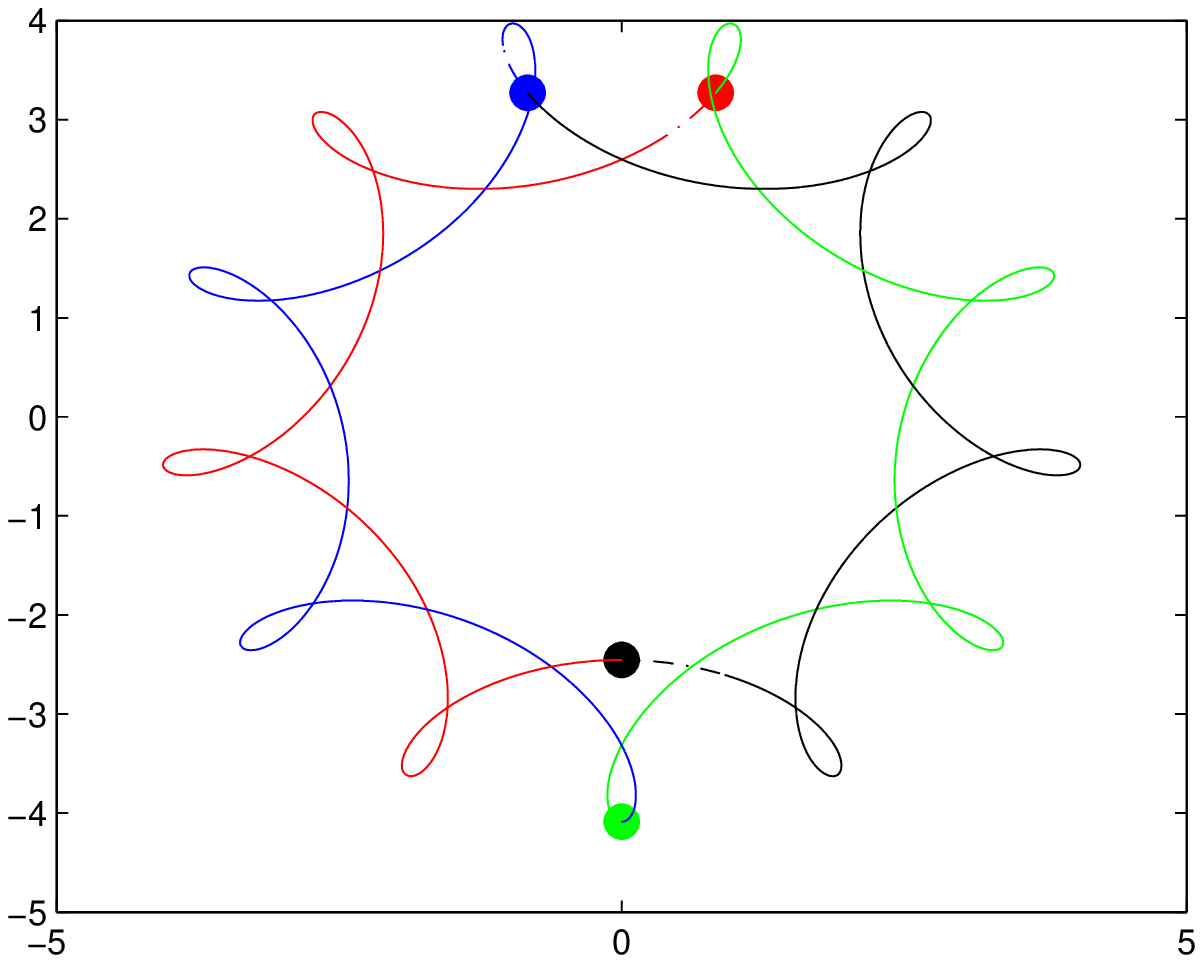}
\includegraphics[height=5cm,width=.32\textwidth]{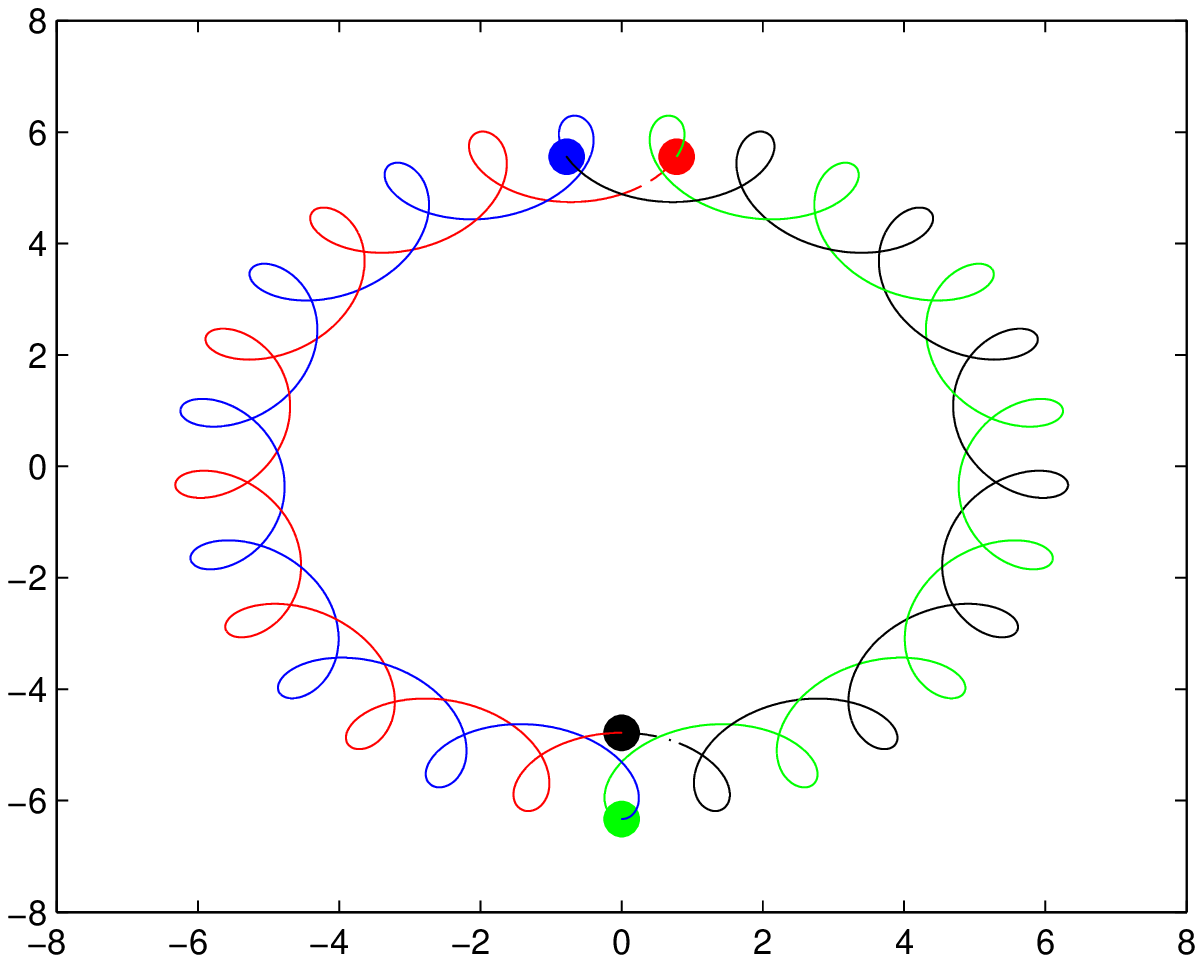}
\caption{\small Simple Choreographic Solutions for $\theta = \frac{P}{Q}\pi$ with $Q\equiv 1 \mod 4$. Bodies chase each other in the order $q_1 (red) $ $\rightarrow q_2 (black)$ $ \rightarrow q_3 (blue)$ $ \rightarrow q_4 (green)$ $ \rightarrow q_1 $. From left to right $\theta= \frac{4\pi}{9},$ $\frac{6\pi}{13},$ and  $\frac{14\pi}{29}$.  }
\label{fig6}\end{figure}

\begin{figure}
\includegraphics[height=5cm,width=.32\textwidth]{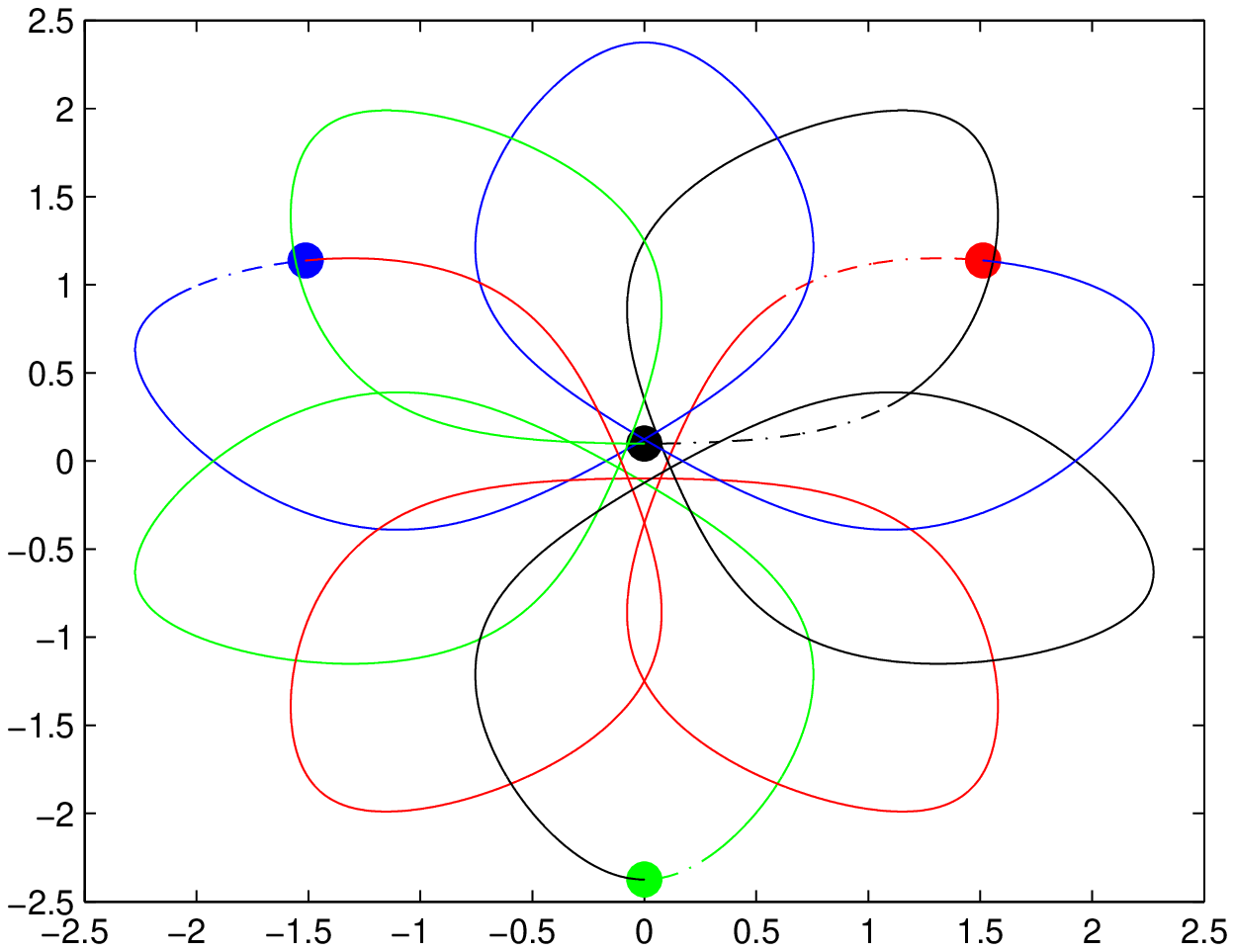}
\includegraphics[height=5cm,width=.32\textwidth]{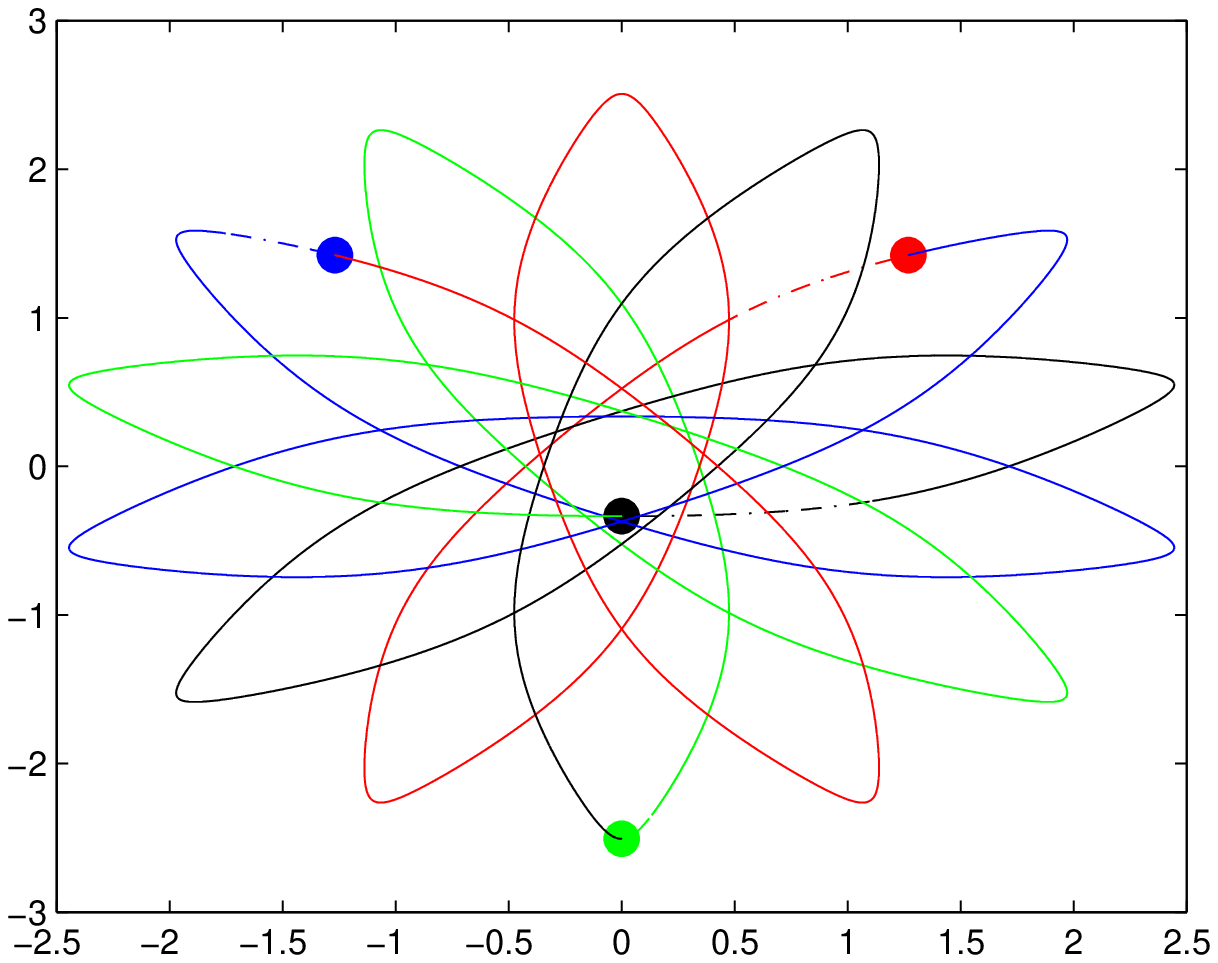}
\includegraphics[height=5cm,width=.32\textwidth]{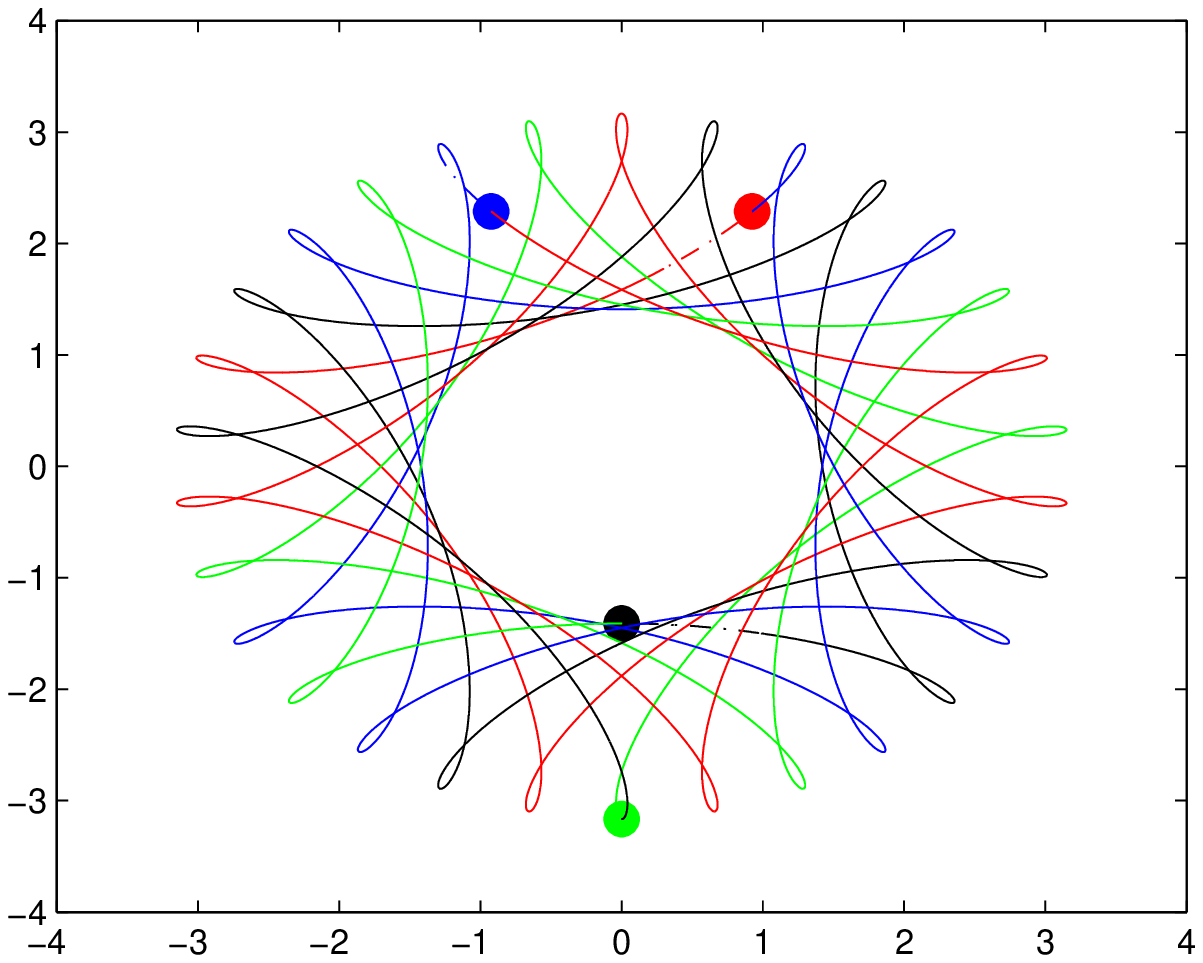}
\caption{\small Double-Choreographic Solutions for $\theta = \frac{P}{Q}\pi$ with $Q\equiv 2 \mod 4$. Bodies chase each other pairwisely as $q_1 (red) \rightarrow q_3 (blue) \rightarrow q_1$ and $q_2 (black) \rightarrow q_4 (green)\rightarrow q_2$. From left to right $\theta= \frac{3\pi}{10},$ $\frac{5\pi}{14},$ and  $\frac{13\pi}{30}$.  }
\label{fig7}\end{figure}

\begin{figure}
\includegraphics[height=5cm,width=.32\textwidth]{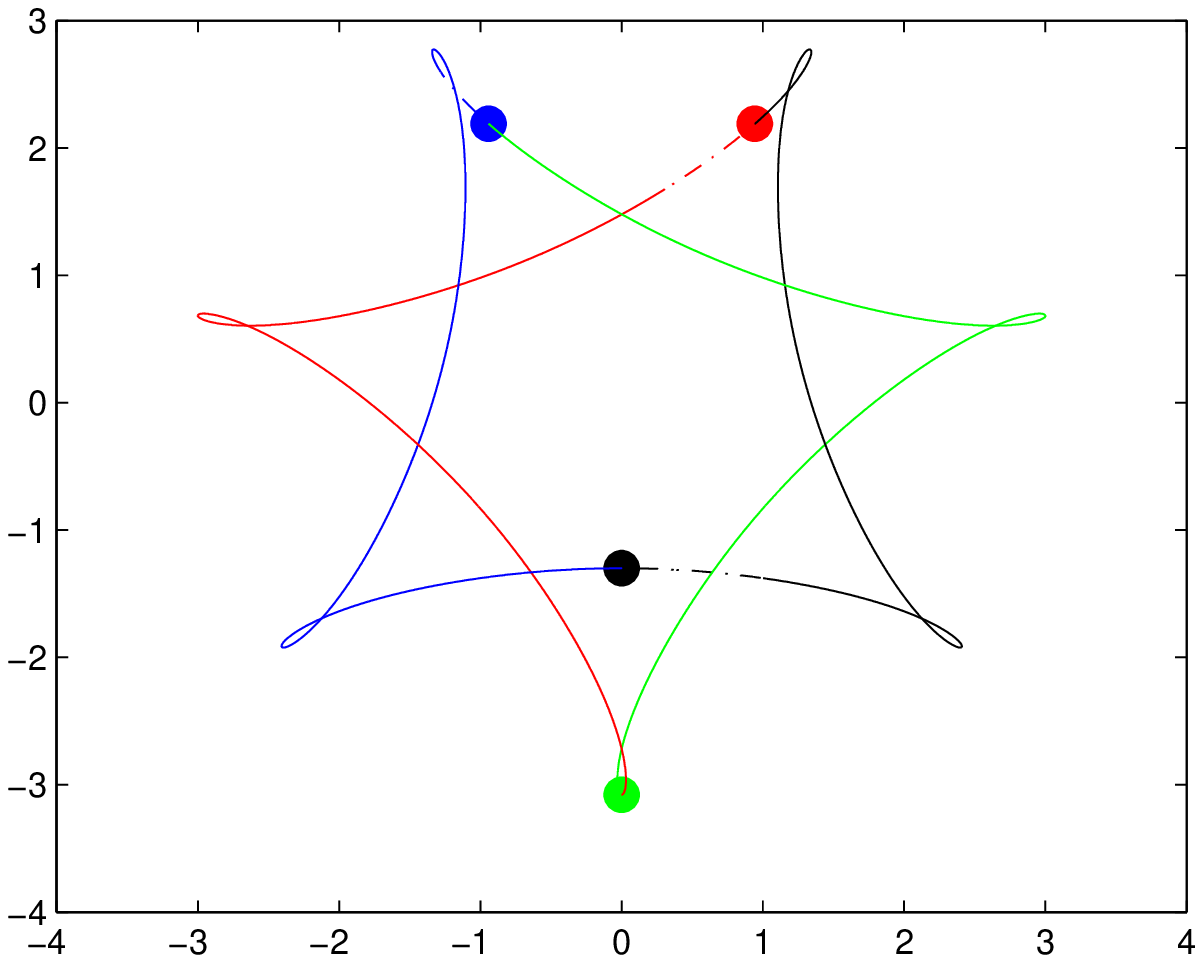}
\includegraphics[height=5cm,width=.32\textwidth]{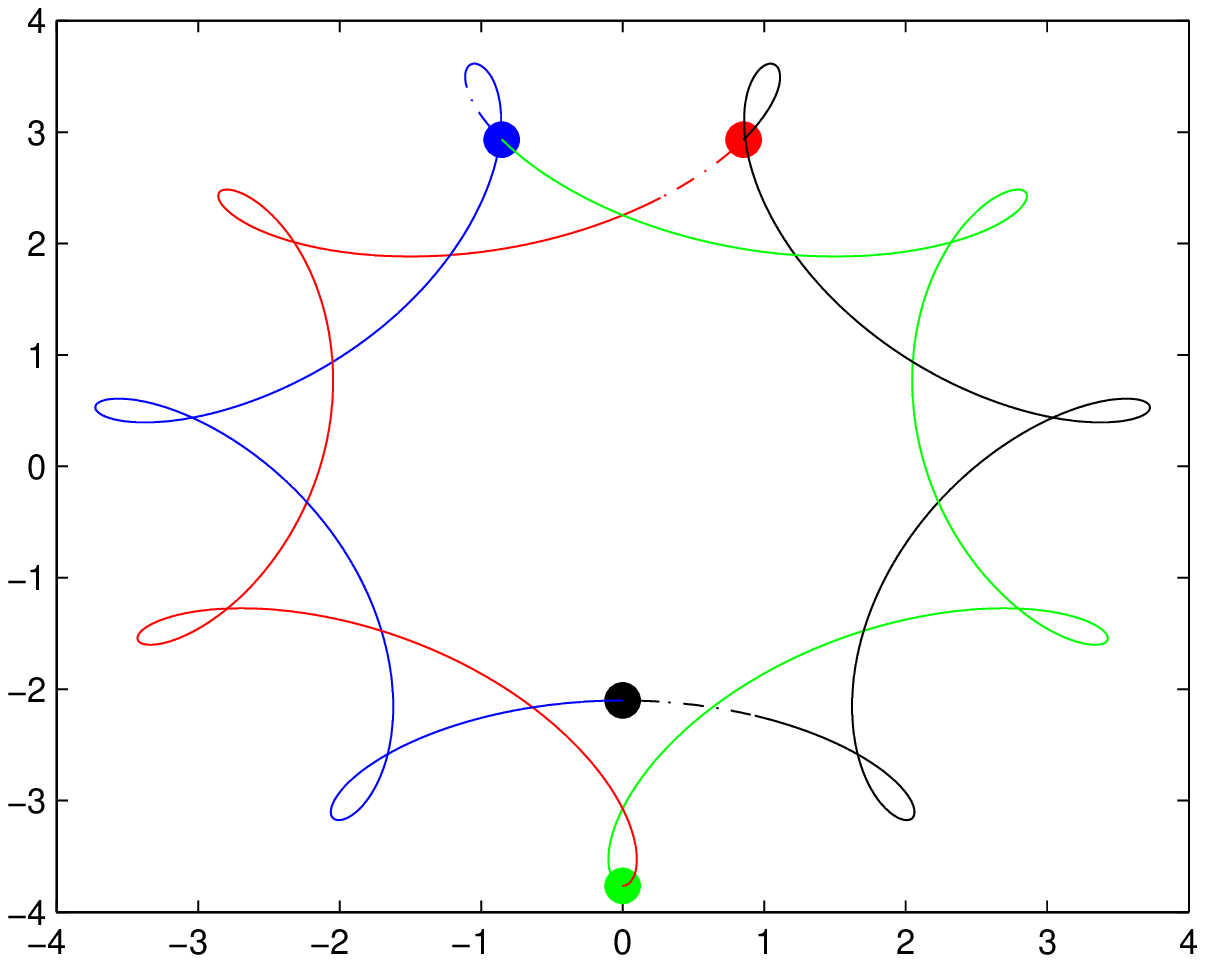}
\includegraphics[height=5cm,width=.32\textwidth]{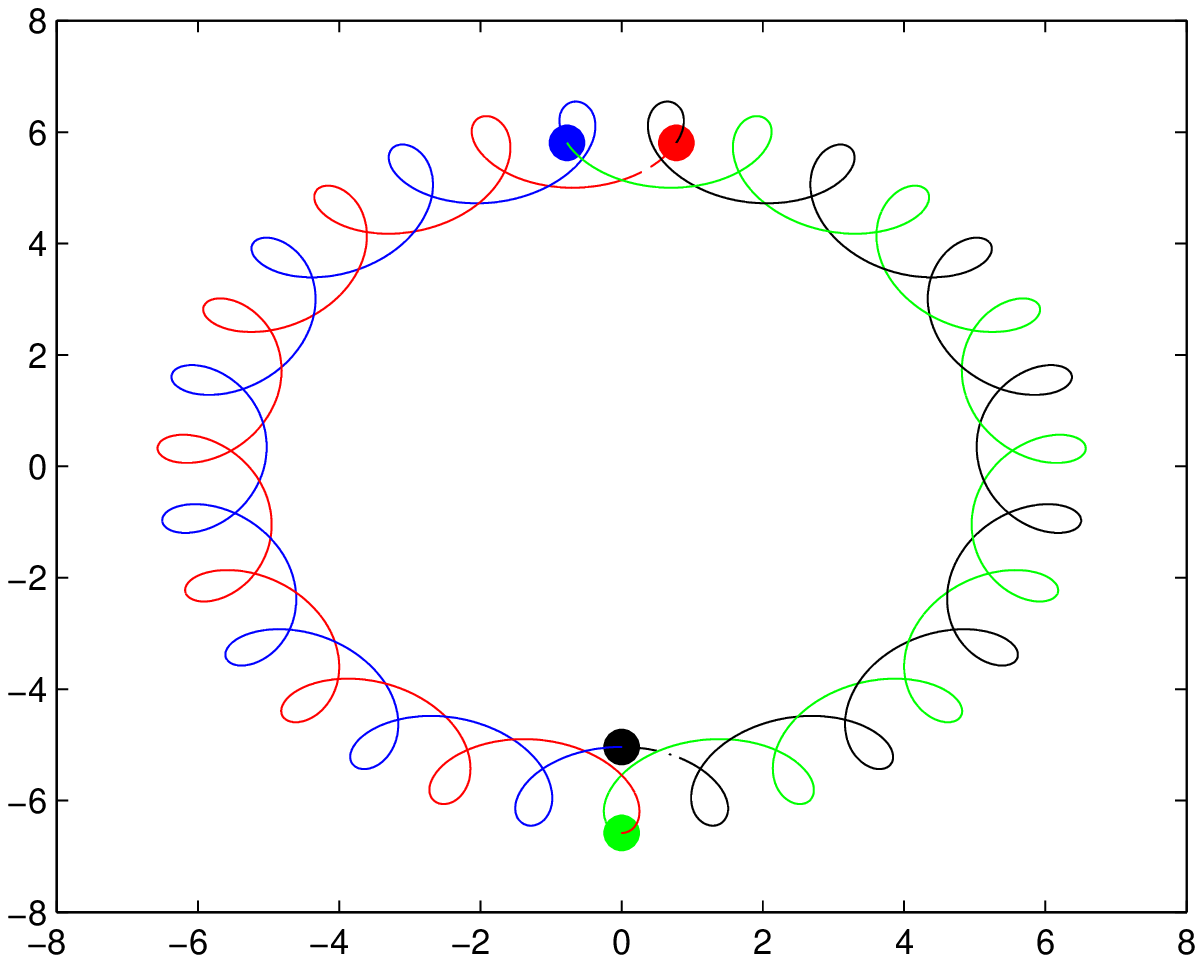}
\caption{\small Simple Choreographic Solutions for $\theta = \frac{P}{Q}\pi$ with $Q\equiv 3 \mod 4$. Bodies chase each other in the order $q_1 (red) \rightarrow q_4 (green) \rightarrow q_3 (blue) \rightarrow q_2 (black)\rightarrow q_1$. From left to right $\theta= \frac{3\pi}{7},$ $\frac{5\pi}{11},$ and  $\frac{15\pi}{31}$.  }
\label{fig8}\end{figure}

\begin{figure}
\includegraphics[height=5cm,width=.32\textwidth]{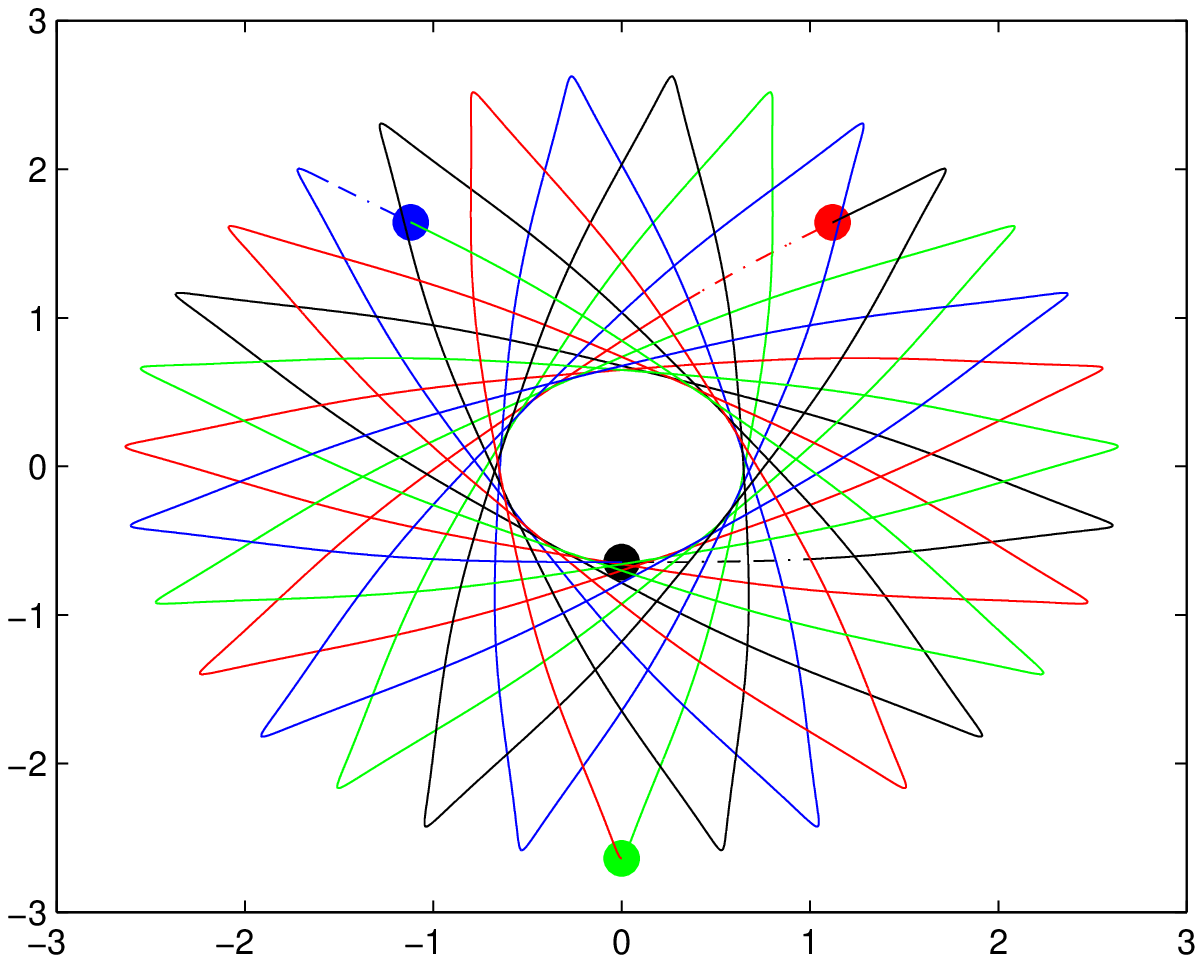}
\includegraphics[height=5cm,width=.32\textwidth]{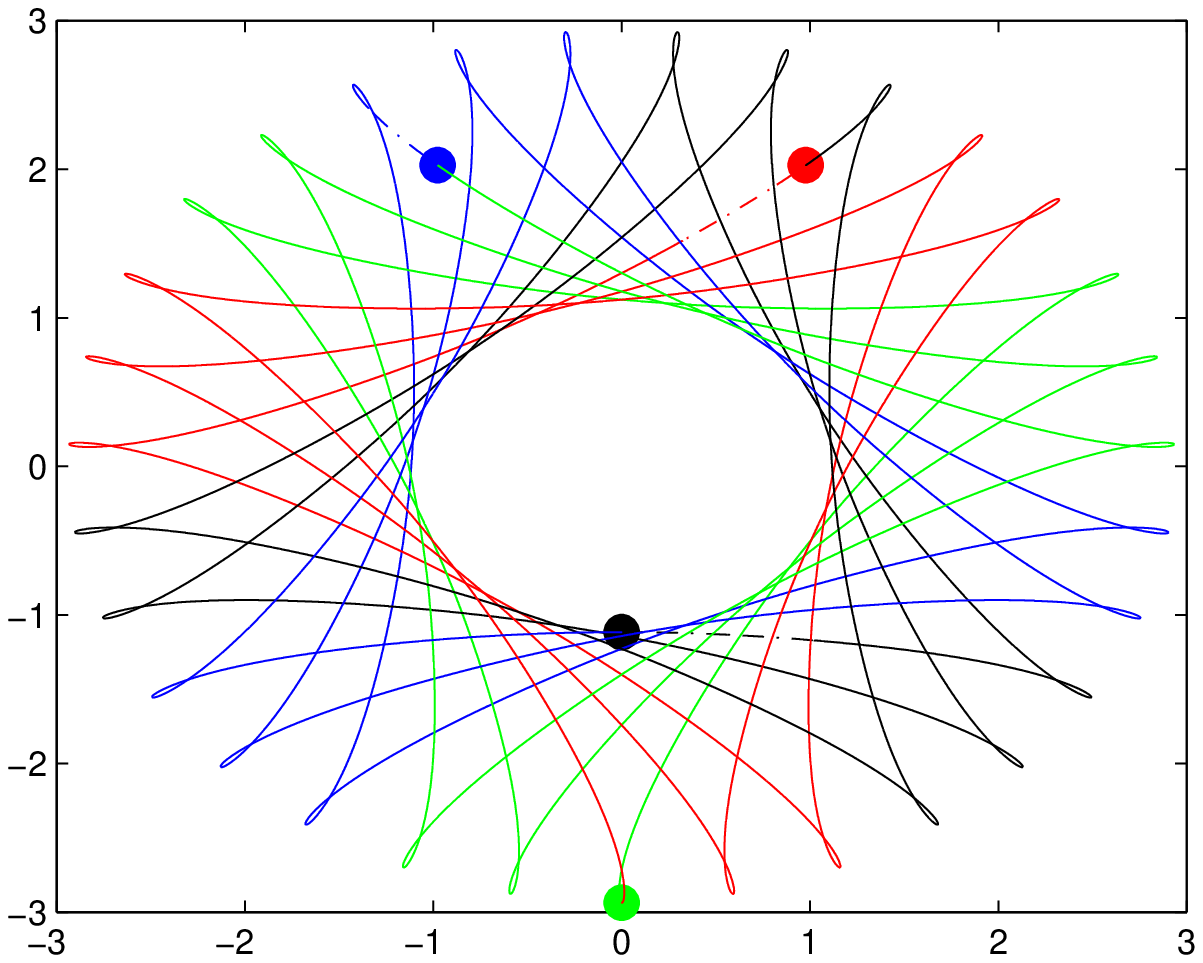}
\includegraphics[height=5cm,width=.32\textwidth]{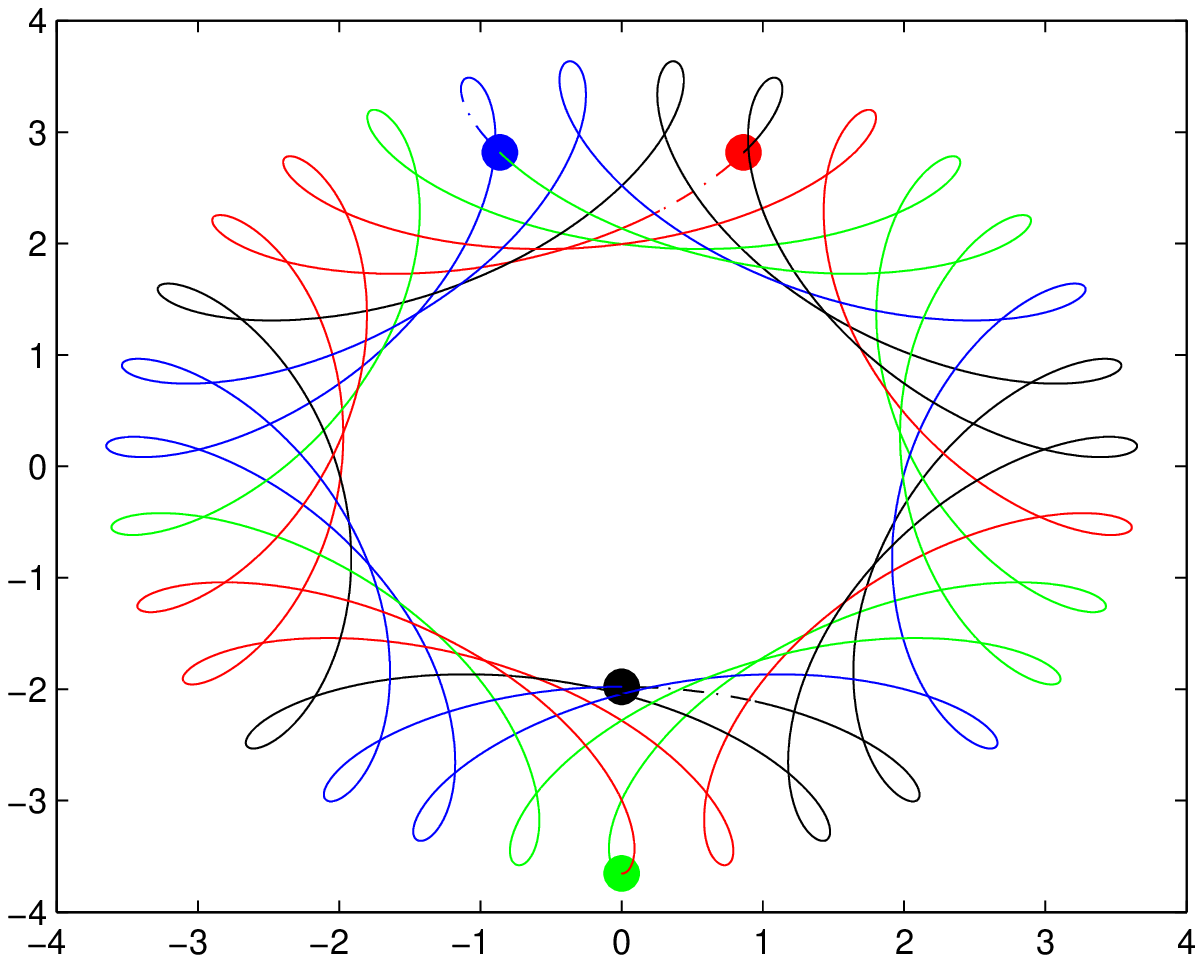}
\caption{\small Periodic solutions with same periods $\mathcal{T}=248T$. From left to right $\theta= \frac{12\pi}{31},$ $\frac{13\pi}{31},$ and  $\frac{14\pi}{31}$.  }
\label{fig11}\end{figure}
\begin{figure}
\includegraphics[height=5cm,width=.32\textwidth]{pathN15f31.eps}
\includegraphics[height=5cm,width=.32\textwidth]{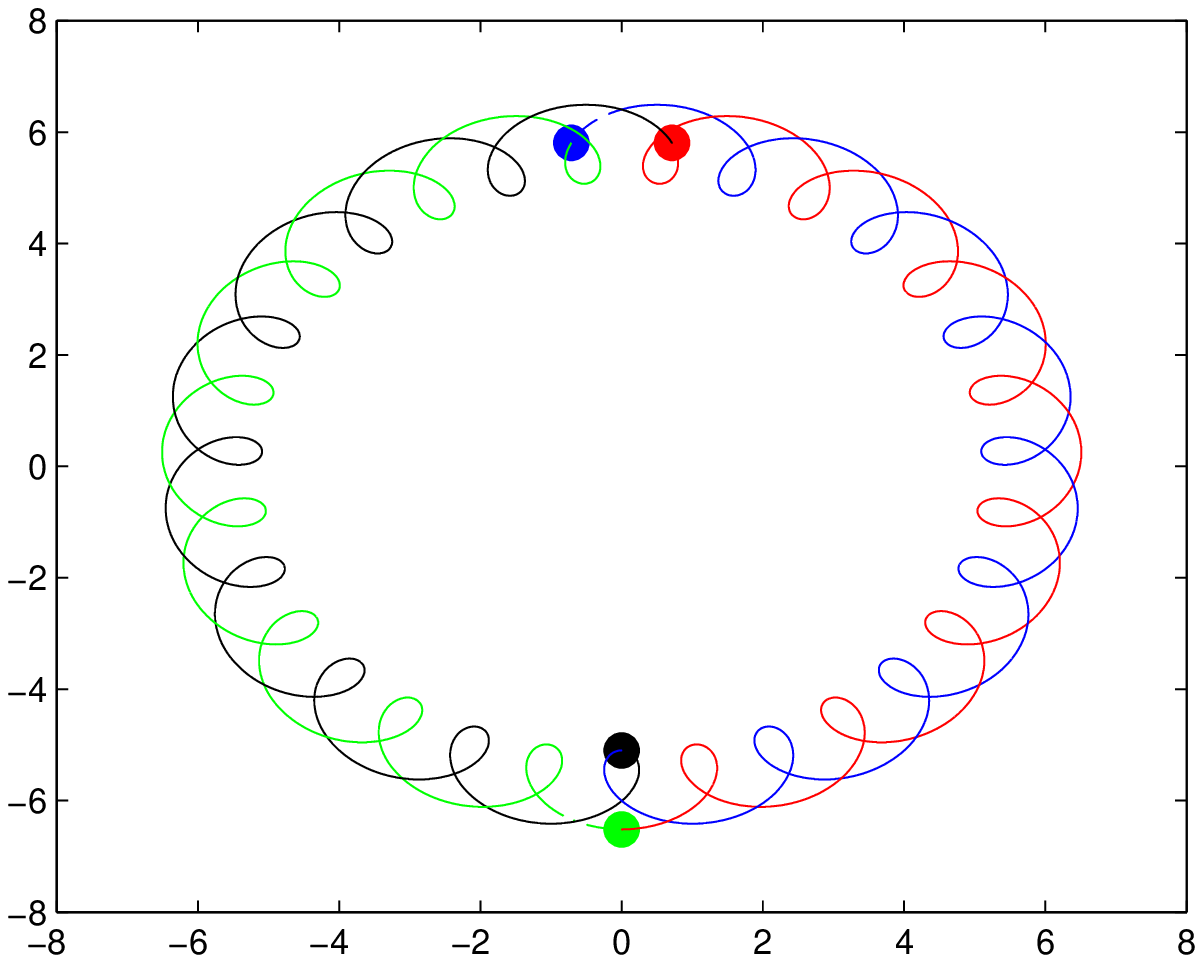}
\includegraphics[height=5cm,width=.32\textwidth]{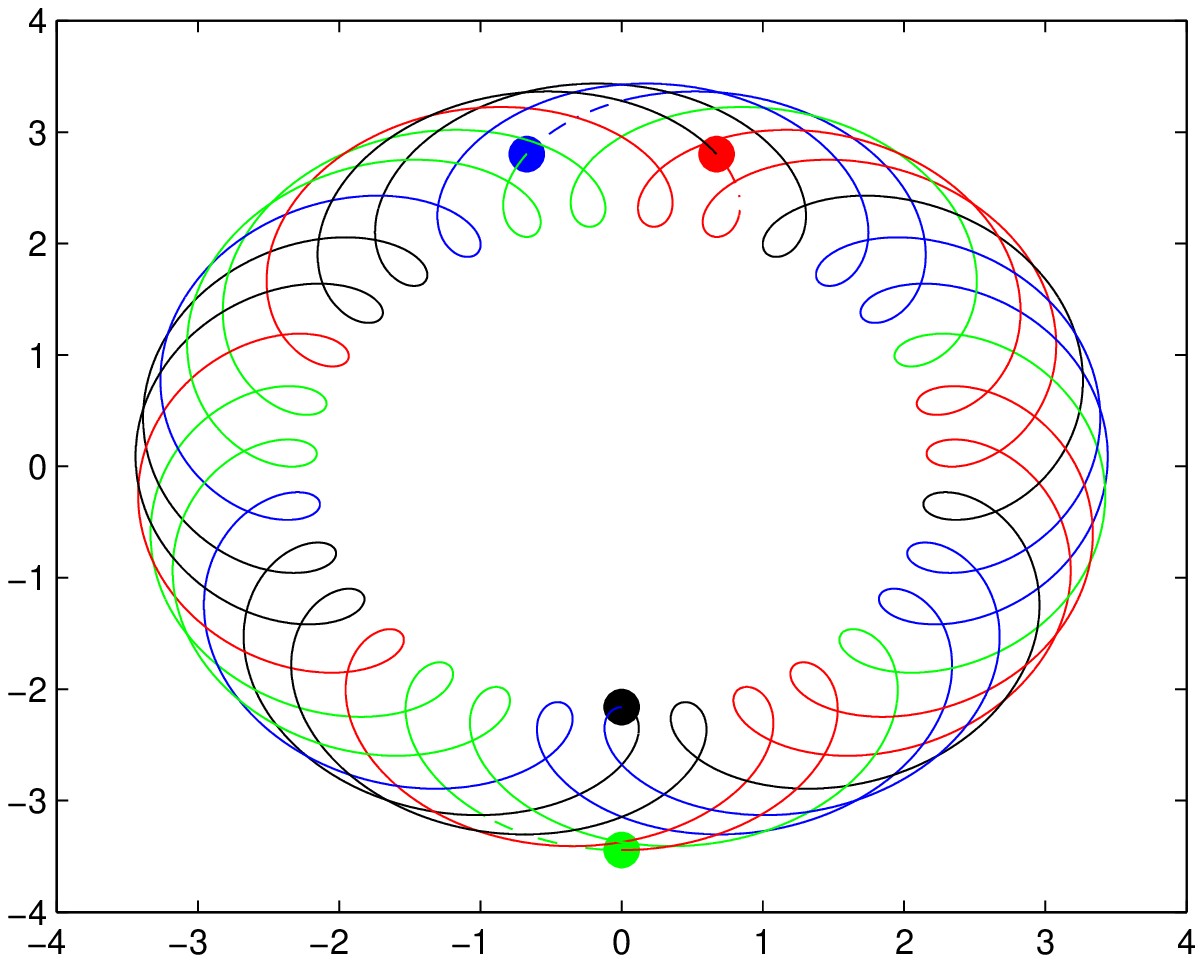}
\caption{\small Periodic solutions with same periods $\mathcal{T}=248T$. From left to right $\theta= \frac{15\pi}{31},$ $\frac{16\pi}{31},$ and  $\frac{17\pi}{31}$.  }
\label{fig11}\end{figure}

\begin{figure}
\includegraphics[height=5cm,width=.32\textwidth]{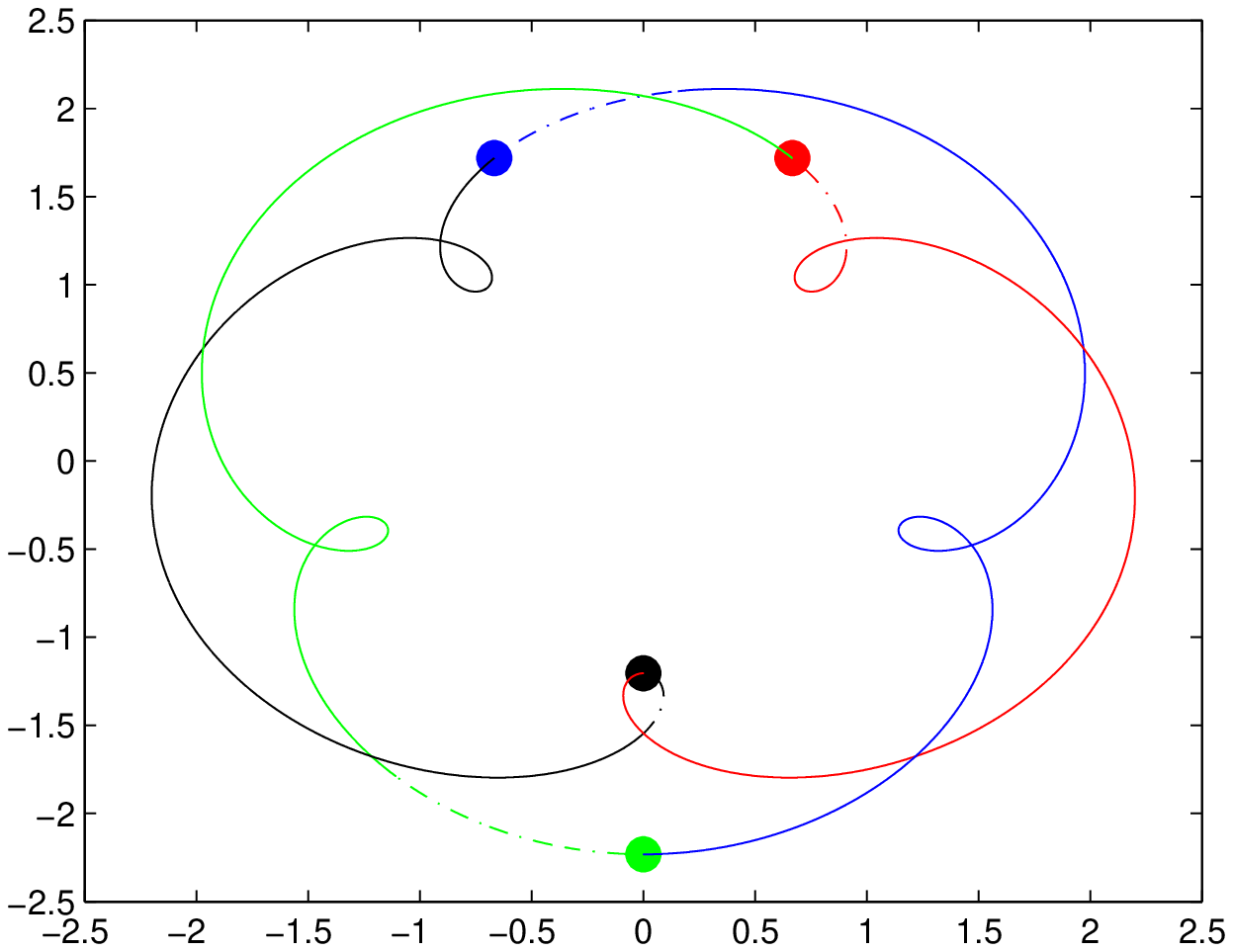}
\includegraphics[height=5cm,width=.32\textwidth]{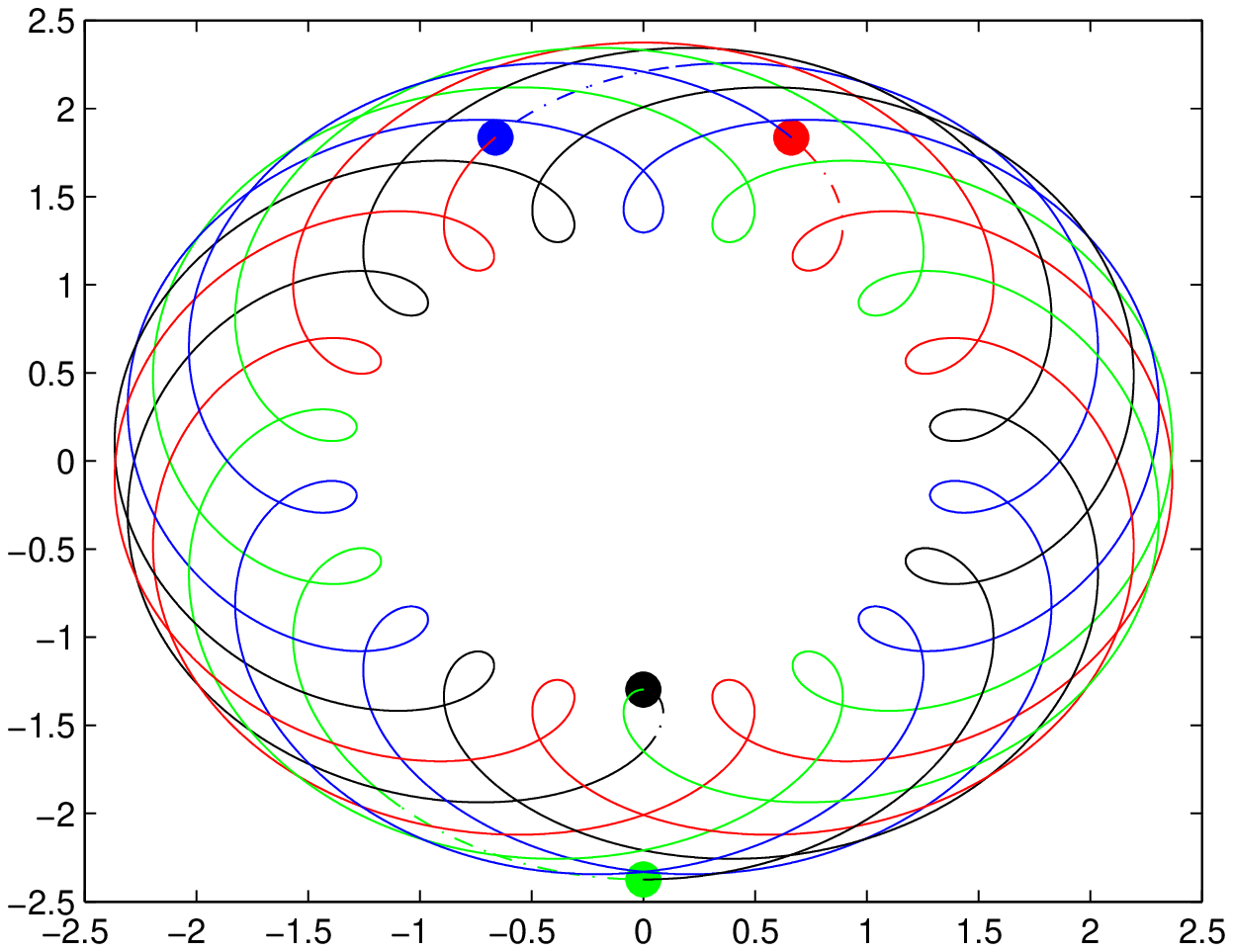}
\includegraphics[height=5cm,width=.32\textwidth]{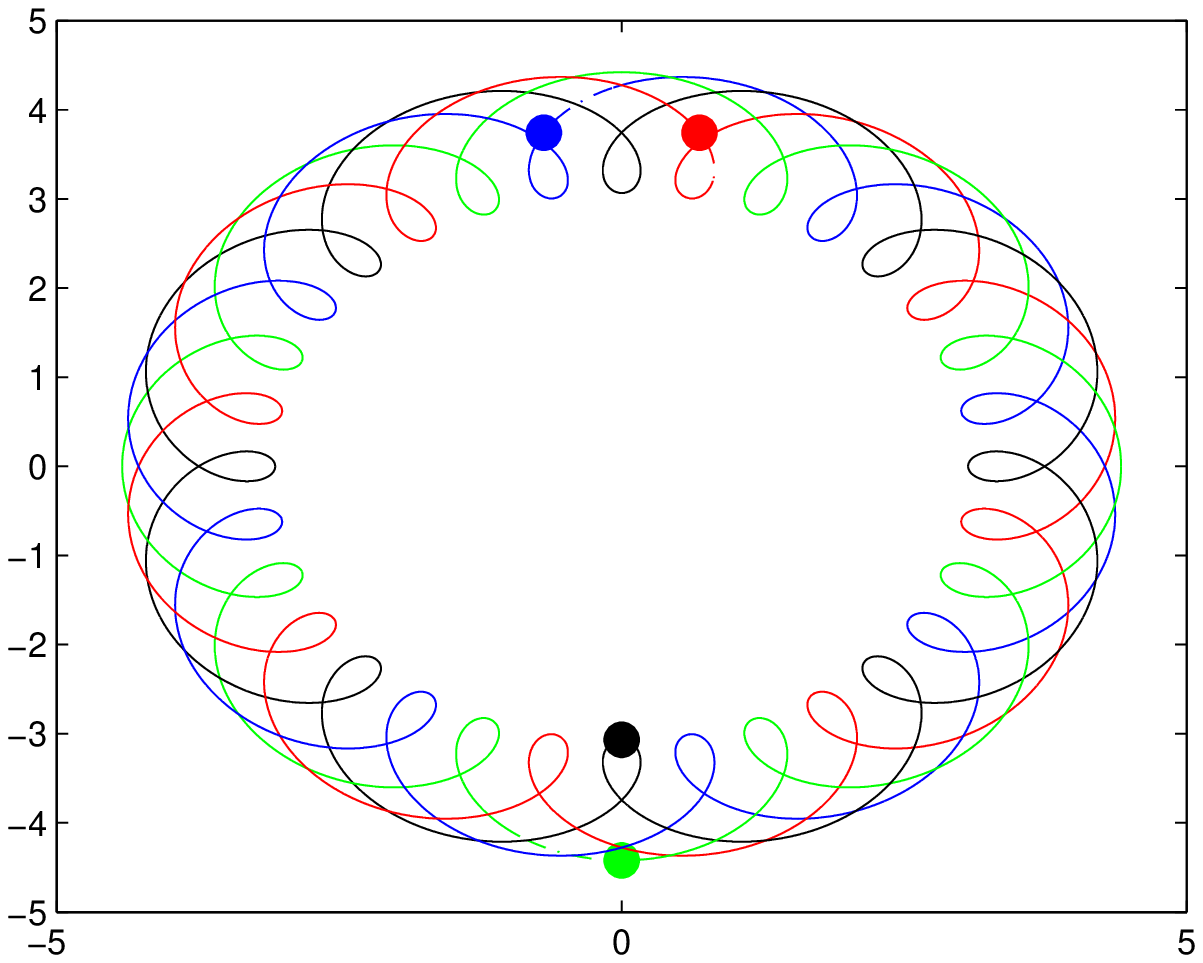}
\caption{\small Periodic solutions for $\theta>\frac{\pi}{2}$. From left to right $\theta= \frac{3\pi}{5},$ $\frac{13\pi}{22},$ and  $\frac{17\pi}{32}$.  }
\label{fig10}\end{figure}

\end{document}